\documentclass[11pt,
]{amsart}

\usepackage{
amssymb, graphicx,  wrapfig
}
\usepackage[font=small,labelfont=bf]{caption} 
\usepackage[subrefformat=parens]{subcaption}

\date{\today}

\newtheorem{theorem}{Theorem}[section]
\newtheorem{proposition}{Proposition}[section]

\newtheorem{definition}{Definition}[section]
\newtheorem{corollary}{Corollary}[section]

\theoremstyle{definition}
\newtheorem{remark}{Remark}[section]

\DeclareMathOperator{\sinc}{sinc}
\DeclareMathOperator{\Vol}{Vol}

\DeclareMathOperator{\supp}{supp}

\DeclareMathOperator{\WF}{WF}

\DeclareMathOperator{\WFH}{WF_{\it h}}

\newcommand{\eps}{\varepsilon}

\renewcommand{\sc}{semi\-classical}
\newcommand{\R}{{\bf R}}
\DeclareMathOperator{\Id}{Id}
\renewcommand{\r}[1]{(\ref{#1})}
\renewcommand{\Xi}{B}
\newcommand{\PDO}{$\Psi$DO}
\newcommand{\HPDO}{$h$-$\Psi$DO}
\newcommand{\be}[1]{\begin{equation}\label{#1}}
\newcommand{\ee}{\end{equation}}

\renewcommand{\d}{\mathrm{d}}

\renewcommand{\i}{\mathrm{i}}

\newcommand{\bo}{\partial \Omega}

\newcommand{\B}{B}

\title[Semiclassical Sampling]{Semiclassical Sampling and discretization of certain linear inverse problems}

\author[Plamen Stefanov]{Plamen Stefanov}
\address{Department of Mathematics, Purdue University, West Lafayette, IN 47907}

\thanks{Partially supported by the National Science Foundation under
grant DMS-1600327.}

\begin{document}
\begin{abstract}
We study sampling of Fourier Integral Operators $A$ at rates $sh$ with $s$ fixed and $h$ a small parameter. We show that the Nyquist sampling limit of $Af$ and $f$ are related by the canonical relation of $A$ using semiclassical analysis. We apply this analysis to the Radon transform in the parallel and the fan-beam coordinates. We explain and illustrate the optimal sampling rates for $Af$, the aliasing artifacts, and the effect of averaging (blurring) the data $Af$. We prove a Weyl type of estimate on the minimal number of sampling points to recover $f$ stably in terms of the volume of its \sc\ wave front set.
\end{abstract} 
\maketitle

\section{Introduction}  
The classical Nyquist--Shannon sampling theorem says that a function $f\in L^2(\R^n)$ with a Fourier Transform $\hat f$ supported in the box $[-\Xi, \Xi]^n$ can be uniquely and stably recovered from its samples $f(sk)$, $k\in\mathbf{Z}^n$ as long as $0<s\le \pi/\Xi$. More precisely, we have
\be{1}
f(x) = \sum_{k\in\mathbf{Z}} f(sk)\chi\Big(\frac{\pi}{s}(x-sk)\Big), \quad \chi(x):= \prod_{j=1}^n\sinc(x^j)
\ee
and 
\be{2}
\|f\|^2 = s^n\sum_{k\in\mathbf{Z}} |f(sk)|^2,
\ee
where $\|\cdot\|$ is the $L^2$ norm, see, e.g., \cite{Natterer-book} or \cite{Epstein-book}. 
If $s<\pi/\Xi$ (strictly) then we have oversampling and one can replace the $\sinc$ function in \r{1} by   a faster decaying one, see Theorem~\ref{thm_sampling} below. 
For practical purposes, there are two major inconveniences: we need infinitely many samples and $f$ has to be real analytic, and in particular,  it cannot be compactly supported unless it is zero. The stability \r{2} allows us to resolve those difficulties by using approximate recovery for approximately band limited functions. Let us say that $f$ is ``essentially supported'' in some box $[-R,R]^n$ in the sense that $f=f_0+f_1$ with $\|f_1\|\le \eps_1\ll1$ and $f_0$ being ``essentially $\Xi$-band limited'' in the sense that the $L^2$ norm of $\hat f$ outside that frequency box is bounded by some $0<\eps_2\ll1$. Then  \r{1} recovers $f_0$ up to an error small with $\eps_2$, see  \cite{Natterer-book}, by sampling $f_0$ (not $f$). The effect of replacing $f$ by $f_0$ can be estimated in terms of $\eps_1$ as well. 
 
There are  generalizations of the sampling theorem to non-rectangular but still periodic grids, see, e.g., \cite{PetersenM} or to some non-uniform ones, see, e.g., \cite{Landau-sampling} but the latter theory is not as complete when $n\ge2$. The version presented above is equivalent to viewing $\R^n$ as a product of $n$ copies of $\R$. In particular, it is invariant under translations and dilations and has a natural extension to actions of linear transformations. On the other hand, the conditions are sharp both for uniqueness and for stability. If the sampling rate (Nyquist) condition is violated, there is non-uniqueness and if we still use \r{1}, we get aliasing. There are also versions for $f$  belonging to spaces different than $L^2$. The proof of the sampling theorem is equivalent to thinking about $f$ as the inverse Fourier transform of $\hat f$, the latter compactly supported. Therefore the samples $f(sk)$ are essentially the Fourier coefficients of $\hat f$ extended as a $2\pi/s$ periodic function in each variable (which also explains the Nyquist limit condition), see Theorem~\ref{thm_FIO}.

The purpose of this work is to study the effect of sampling the data at a certain rate for a class of  linear inverse problems. This class consists of problems of inverting a Fourier Integral Operator (FIO): find $f$ if
\be{3}
Af=m
\ee
with $m$  given (so far, noiseless) and $A$ is an FIO of a certain class. There are many examples: inversion of the Euclidean X-ray and the Radon transforms, for which the sampling problem is well studied, see, e.g., the references in \cite[Ch.~III]{Natterer-book}; inversion of the geodesic X-ray transform and more general Radon transforms; thermo and photo-acoustic tomography with a possibly variable speed, etc. A large class of integral geometry operators are in fact FIOs, as first noticed by Guillemin \cite{Guillemin85,GuilleminS}. The solving operators of hyperbolic problems are also FIOs in general. 
On the other hand, many non-linear inverse problems have a linearization of this kind, like the boundary rigidity problem or various problems of recovery coefficients in a hyperbolic equation from boundary measurements. 
 
 We study the following types of questions.
%
 
\textbf{(i) Sampling $Af$:} Given an essential frequency bound of $f$ (the lowest possible ``detail''), how fine should we sample the data $Af$ for an accurate enough recovery? This question, posed that way, includes the problem of inverting $A$ in the first place, in addition to worrying about sampling. The answer is specific to $A$ which could be associated to a canonical graph or not, elliptic or not, injective or not. 
 Then a reformulation of the first question is --- 
if $f$ is approximately band limited, is also $Af$ approximately band limited, with what limit, and then  what sampling rate will recover reliably $Af$? The problem of recovery of $f$ after that depends on the specific $A$. 

\textbf{(ii) Resolution limit on $f$ given the sampling rate of $Af$.} 
Suppose we have fixed the sampling rate of $Af$ (not necessarily uniformly sampled). In applications, we may not be able to sample too densely. What limit does this pose on  the smallest detail of $f$ we can recover? The answer may depend on the location and on the direction of those details. 

\textbf{(iii) Aliasing.} Above, if $f$ has detail smaller than that limit, there will be aliasing. How will the aliasing artifacts look like? Aliasing is well understood in classical sampling theory but the question here  is what kind of artifacts an aliased $Af$  would create in the reconstructed $f$.

\textbf{(iv) Averaged measurements/anti-aliasing.} Assume  we cannot sample $Af$ densely enough or assume that  $f$ is not even approximately band limited. 
Then the data would be  undersampled. The next practical question is --- can we blur the data \textit{before we sample} to avoid aliasing and then view this as an essentially properly sampled problem but for a blurred version of $f$? This is a standard technique in imaging and in signal processing  but here we want to relate to reconstruction of $f$ to the blurring of the data $Af$. 
In X-ray tomography, for example, this would mean replacing the X-ray with thin cylindrical packets of rays and/or using detectors which could average over small neighborhoods and possibly overlap. One could also vibrate the sample during the scan.  In thermo and photoacoustic tomography, one can take detectors which average over some small areas. The physical  detectors actually do exactly that in order to collect good signal which brings us to another point of view --- physical measurements are actually already averaged and we want to understand what this does to the reconstructed $f$.

To answer the sampling   question (i), one may try to estimate the essential support (the ``band'') of $\widehat{Af}$ given that of $\hat f$ and then apply some of the known sampling theorems.  This is a possible approach for each particular problem but will also require the coefficients of $A$, roughly speaking, to be also band limited, and the band limit of $Af$ would depend of that of $f$ and on $A$. The  operators of interest have singular Schwartz kernels however. Also, it may not be easy to get sharp constants.  One might prove that if $ f$ is approximately $\Xi$ band limited, then $Af$ is, say, $C\Xi$ approximately band limited. The success of this approach would depend heavily on having a sharp constant $C$. Proving that there exists some $C>0$ even with some rough estimate on it would not be helpful because the required sample step $s$ would have to be scaled as $s/C$. In some symmetric cases, such direct approach can and has been done, for example for the Euclidean X-ray/Radon transform, see the references in Section~\ref{sec_R_PG}. Our main interests however is in inverse problems without symmetries (coming from differential equations with variable coefficients, for example). Note that one may try to use interpolation by various functions, like splines, for example but the same problem exists there since the bounds of the error depend on a priori bounds of some higher order derivatives, and the constants in those estimates matter. 
 
To overcome this difficulty, we look at the problem as an asymptotic one. We think of the highest frequency of $f$ (in some approximate sense) as a large parameter 
and we are interested in the optimal sampling rate for $Af$ when that upper bound gets higher and higher; which would force the sampling step $s>0$ to get smaller and smaller. To model that, we rescale 
the dual variable $\xi$ to  $\xi/h$, where $0<h\ll1$ is a small parameter; which would rescale the sampling rate to $sh$. Then we assume that $|\xi|$ is bounded by some constant $B$ which we call a \sc\ band limit of $f$. 
We think of $f$ as a family, depending on $h$. The natural machinery for this is the semi-classical calculus. The frequency content of $f$ locally is described by its semi-classical wave front set $\WFH(f)$: if the latter consists (essentially) of $(x,\xi)$ with $|\xi|\le \Xi$ (then we say that $f$ is \sc ly $B$-band limited), then $\hat f$ is essentially supported in $\Xi/h$. In classical terms, this means $|\hat f(\xi)|=O(h^\infty)$ for $|\xi|\ge B/h$, i.e., $\hat f$ is essentially supported in $|\xi|\le B/h$ as $h\ll1$. One can   think of $B/h$ as the upper bound of $|\xi|$ for $\xi\in\supp \hat f$ (up to a small error). One can also handle  $O(h)$ errors instead of $O(h^\infty)$ by replacing $\WFH$ with is \sc\ Sobolev version.
 If $A$ is a \sc\ FIO (h-FIO), then $\WFH(f)$ is mapped to $\WFH(Af)$ by the canonical relation of $A$. This is also true, away from the zero frequencies, if $A$ is a classical FIO. That property is sharp when $A$ is elliptic (which happens for most stable problems). Therefore, we can estimate sharply $\WFH(Af)$ and apply an appropriate sampling theorem. The canonical relation is typically described by some properties of the geometry of the problem, as we will demonstrate on some examples. 

Knowing $\WFH(Af)$ for all $f$ \sc ly B-band limited $f$ determines the sampling rate for $Af$. Indeed, let $\Sigma_h(Af)$ be the \sc\ frequency set of $Af$ defined as the projection of $\WFH(Af)$ onto the dual variable. Then an upper bound of the size, or even the shape of $\Sigma_h(f)$ determines a  sharp sampling rate for $Af$, see Theorem~\ref{thm_sc}. Since $\WFH(Af)$ is $2n$-dimensional and $\Sigma_h(Af)$ is $n$-dimensional, the latter discards useful information about the $x$-localization of $\WFH(Af)$. Our analysis allows to formulate results about non-uniform but still a union of locally uniform sampling lattices allowing us to use coarser sampling where the frequencies cannot reach their global maximum. We demonstrate how this works for the Radon transform. Note that this non-uniform sampling is not a consequence of a priori assumptions on the localization of $\WFH(f)$ (although, if we have such assumptions, we can do further reductions) --- it depends on the intrinsic geometry of the problem i.e., on the Lagrangian of $A$. 

To answer (ii), assuming that we sample $Af$ at a rate requiring, say $\Sigma_h(Af)$ to be supported in $\{\eta;\; |\eta_j|\le B, \,\forall j\}$, we need to map this box back by the inverse of the canonical relation of $A$. 

The aliasing question (iii) admits a neat characterization. Aliasing is well understood in principle as frequencies $\xi$ shifting (or ``folding'') in the Fourier domain. This can be regarded as a h-FIO, call it $S$, in our setting. If $A$ is associated with a local  canonical diffeomorphism, then inverting $A$ with aliased measurements results in $A^{-1}SA$, which is an h-FIO with a canonical relation a composition of the three. While classical aliasing shifts  frequencies but preserves the space localization (see, e.g., Figure~\ref{fig_aliasing}); we get a new effect here: the ``inversion'' $A^{-1}SA$ does not preserve the space localization and can shift parts of the image, see Figure~\ref{fig_R_aliasing2}. 
 
The averaged measurements/anti-aliasing problem (iv) can be resolved as follows. If the a priori estimate of the size of $\WFH(f)$ is too high (or infinite) for our sampling rate, we can blur the data before sampling, i.e., apply an anti-aliasing filter, which is routinely done in signal and image processing. To do this, take some $\phi\in C^\infty$ decaying fast enough, set $\phi_h = h^{-m}\phi(\cdot/h)$ (here $m$ is the dimension where we collect the measurements, and $h^{-m}$ is a normalization factor), and consider $\phi_h *Af$. That convolution can be seen to be a \sc\ \PDO\ (a Fourier multiplier) with a \sc\ symbol $\hat\phi$. 
One can also use a more general \sc\ \PDO. By Egorov's theorem, $\phi_h *Af = AP_hf+O(h^\infty)f$, where $P_h$ is a zeroth order \sc\ \PDO\ with a principal symbol obtained by $\hat\phi$ pulled back  by the canonical relation of $A$. The essential support of the full symbol is also supported where the principal one is. Therefore, if $A$ is associated to a diffeomorphism at least (but not only), one can choose $\phi_h$ so that $P_h$ plays the role of the low-pass filter needed for proper sampling if the rate of the latter on the image side is given. 
If the inverse problem is well posed in a certain way, we will recover $P_hf$ stably and the latter will be a \sc\ \PDO\ applied to $f$ cutting off higher frequencies which can be viewed as $f$ regularized. We can even choose $P_h$ as desired, and then compute $\phi$ which is general may change from one sample to another. This brings us to the more general question of sampling $Q_hAf$, where $Q_h$ is an \HPDO\ limiting the frequency content, instead of being just a convolution. The analysis is then similar. 

Finally, we prove in Section~\ref{sec_NU} an asymptotic lower bound on the number of non-uniform  sampling points needed to sample stably $f$ with $\WFH(f)$ in a given compact subset $\mathcal{K}$ of $T^*\R^n$. It is of Weyl type and equal to $ (2\pi h)^{-n} \Vol(\mathcal{K}^\text{int})$. This generalizes a theorem by Landau \cite{Landau-sampling} to our setting where the sampling is classical and the number of the sampling points in any $\Omega$ is estimated by below by $(2\pi)^{-n}\Vol(\Omega\times \mathcal{B})$, if $\supp \hat f\subset  \mathcal{B}$. 
 
We want to 
mention that  numerical computations of FIOs by discretization is an important problem by itself, see, e.g., \cite{ Candes_09, Anderson_12, DeHoop_13,Caday_16} which we do not study here. The emphasis of this paper however is different: how the sampling rate and/or the local averaging of the FIO affect the amount of microlocal data we collect and in turn how they could limit or not its microlocal inversion. 
 
\textbf{Acknowledgments.} The author  thanks Fran\c{c}ois Monard for the numerous discussions on both the theoretical and on the numerical aspects of this project;  Yang Yang for allowing him to use his code for generating the numerical examples in Figure~\ref{TAT_pic1}; and Maciej Zworski and Kiril Datchev for discussions on the relationship between classical and \sc\ \PDO s and FIOs. 

\section{Action of \PDO s and FIOs on the \sc\ wave front set} 
\subsection{Wave Front sets} Our main reference for the \sc\ calculus is \cite{Zworski_book}, see also \cite{DZ-book}. For the sake of simplicity, we work in $\R^n$ but those notions are extendable to manifolds.  
Recall that the \sc\ Fourier transform $\mathcal{F}_h f$ of a function depending also on $h$  is given by
\[
\mathcal{F}_hf(\xi) =\int e^{-\i x\cdot\xi/h} f(x)\,\d x.
\]
This is just a rescaled Fourier transform $\mathcal{F}_hf(\xi)=\hat f(\xi/h)$. Its inverse is $ (2\pi h)^{-n}\mathcal{F}_h^*$. 
We recall the definition of the semiclassical wave front set of a tempered $h$-depended distribution first. In this definition, $h>0$ can be arbitrary but in \sc\ analysis,  $h\in (0,h_0)$ is a ``small'' parameter and we are interested in the behavior of functions and operators as $h$ gets smaller and smaller. Those functions are $h$-dependent and we use the notation $f_h$ or $f_h(x)$ or just $f$. 
We follow \cite{Zworski_book} with the  choice of the  Sobolev spaces to be the semiclassical ones defined by the norm
\[
\|f\|^2_{H_h^s}= (2\pi h)^{-n}\int \langle\xi\rangle^{2s}|\mathcal{F}_hf(\xi)|^2\,\d\xi. 
\]
Then an $h$-dependent family $f_h\in\mathcal{S}'$ is said to be $h$-tempered (or just tempered) if $\|f_h\|_{H^s_h}=O(h^{-N})$ for some $s$ and $N$. All functions in this paper are assumed tempered even if we do not say so. The \sc\ wave front set of a tempered family $f_h$ is the complement of those $(x_0,\xi^0)\in\R^{2n}$ for which there exists a $C_0^\infty$ function $\phi$ so that $\phi(x_0)\not=0$ so that
\be{WFH}
\mathcal{F}_h(\phi f_h)= O(h^\infty)\quad \text{for $\xi$ in a neighborhood of $\xi^0$}
\ee
in $L^\infty$ (or in any other ``reasonable'' space, which does not change the notion). The \sc\ wave front set naturally lies in $T^*\R^n$ but it is not conical as in the classical case. Note that the zero section can be in $\WFH(f)$.

There is no direct relationship between the \sc\ wave front $\WFH$ and the classical one $\WF$ (when $h$ is fixed in the latter case), see also \cite{Zworski_book}. For example, for $f\in \mathcal{S}$ independent of $h$, $\WFH(f)=\supp f\times\{0\}$ while $\WF(f)$ is empty. On the other hand, if $g$ is singular and compactly supported, then for $f(x)=g(x-1/h)$ we have $\WFH(f)=\emptyset$ while $\WF(f)$ is non-empty for every $h$, see \cite{Zworski_book}. Sj\"ostrand proposed adding the classical wave front set to $\WFH$ by considering the latter in $T^*\R^{n}\cup S^* \R^{n}$, where the second space (the unit cosphere bundle)  represents $T^*\R^n$ as a conic set, i.e., each $(x,\xi)$ with $\xi$ unit is identified with the ray $(x,s\xi)$, $s>0$. Their points are viewed as ``infinite'' ones describing the behavior as $|\xi|\to\infty$ along different directions.  An infinite point  $(x_0,\xi^0)$ does not belong to the so extended $\WFH(f)$ if we have
\be{infWF}
\mathcal{F}_h(\phi f_h)= O(h^\infty\langle \xi \rangle^{-\infty})\quad  \text{for $\xi$ in a conical neighborhood of $\xi^0$}
\ee
with $\phi$ as above. 
Our interest is in functions which are localized in the spatial variable and  do not have infinite singularities. In  \cite{Zworski_book}, it is said that a tempered  $f_h$  is localized in phase space, if there exists $\psi\in C_0^\infty(\R^{2n})$ so that
\be{loc}
(\Id - \psi(x,hD))f_h=O_{\mathcal{S}}(h^\infty),
\ee
see the definition of \HPDO s below. Such functions do not have infinite singularities and are smooth. 
We work with functions localized in phase space and those are the functions which can be sampled properly anyway. In practical applications, this assumption is satisfied by the natural resolution limit of the data we collect, for example the diffraction limit.

Other examples of \sc\ wave front sets are the following. If $f_h=e^{\i x\cdot \xi_0/h}$, then $\WFH(f)=\R^n\times\{\xi_0\}$. The coherent state 
\be{coh}
f_h (x;x_0,\xi_0)= e^{\i x\cdot \xi_0/h - |x-x_0|^2/2h}
\ee
 (to normalize for unit $L^2$ norm, we need to multiply by $(\pi h)^{-n/4}$) satisfies $\WFH(f)= \{(x_0,\xi_0)\}$. Its real or imaginary part have wave front sets at $(x_0,\xi_0)$ and $(x_0,-\xi_0)$. We will use such states in our numerical examples. 

It is convenient to introduce the notation $\Sigma_h(f)$ for the \sc\ frequency set of $f$.
\begin{definition}\label{def_S} For each tempered $f_h$ localized  in phase space, set 
\[
\Sigma_h(f) = \{\xi;\; \text{$\exists x$ so that $(x,\xi)\in\WFH(f)$}\}.
\]
\end{definition}
In other words, $\Sigma_p$ is the projection of $\WFH(f)$ to the second variable, i.e.,
\be{Sigma}
\Sigma_h(f)= \pi_2\circ \WFH(Af),
\ee
where $\pi_2(x,\xi)=\xi$. 
 If $\WFH(f)$ (which is always closed) is bounded and therefore compact, then $\Sigma_h(f)$ is compact.

\begin{definition}\label{def_B}
We say that   $f_h\in C_0^\infty(\R^n)$  is \sc ly band limited (in $\mathcal{B}$),  
if (i) $\supp f_h$ is contained in an $h$-independent compact set, 
(ii) $f$ is tempered,  and (iii) 
there exists a compact set  $\mathcal{B}\subset\R^n$,  so that  for every open $U\supset \mathcal{B}$, we have 
\be{defB}
|\mathcal{F}_h f(\xi)|\le  C_N h^N\langle\xi\rangle^{-N}\quad \text{for $\xi\not\in U$}
\ee
for every $N>0$. 
\end{definition}

If $h$ is fixed, this estimate trivially holds  for every $\xi$. Its significance is in  the $h$ dependence. 
In particular, such functions do not have infinite singularities, and are localized in phase space. In applications, we  take $\mathcal{B}$ to be $[-B,B]^n$ with some $B>0$ or the ball $|\xi|\le B$ or some other set, see, e.g., Figure~\ref{fig_f_gaussian}. 

As an example, for $0\not=\chi\in C_0^\infty$, $f:= \chi(x)e^{\i x\cdot\xi_0/h}$ is \sc ly band limited with $\mathcal{B}=\{\xi_0\}$. Indeed, $\mathcal{F}_hf(\xi) = \hat\chi((\xi-\xi_0)/h)$ decays rapidly for $\xi\not=\xi_0$ as $h\to0$. Clearly, that decay is not uniform as $\xi\to\xi_0$ which explains the appearance of $U$ in the definition. We could have required \r{defB} to hold in the closure of $\R^n\setminus \mathcal{B}$ to avoid introducing $U$; then in this example, one can take $\mathcal{B}$ to be the closure of every neighborhood of $\xi_0$. Both definitions would work fine for the intended applications. To generalize this example, we can take a superposition of such functions with $\xi_0$ varying over a fixed compact set $\mathcal{B}$ to get $f=\chi \mathcal{F}_h^{-1}g$ with $g\in L^1$, $\supp g\subset\mathcal{B}$ to be \sc ly band limited with frequency set in $\mathcal{B}$. 

Another example of \sc ly band limited functions can be obtained by taking any $f\in\mathcal{E}'(
\R^n)$ and convolving if with $\phi_h=h^n\phi(\cdot/h)$ with $\supp\hat\phi\in C_0^\infty$. Then $\phi_h*f$ is \sc ly band limited with $\mathcal{B}=\supp\hat\phi$. 

\begin{proposition}\label{pr_bl} 
Let  $\mathcal{B}\subset\R^n$ be a compact set. For every  tempered $f_h\in C_0^\infty$ with support contained in an $h$-independent compact set, the following statements are equivalent:

(a) $f_h$ is \sc ly band limited, 

(b) $f_h$ is localized in phase space, 

(c) $\WFH(f)$ is finite and compact. 
\end{proposition}

\begin{proof}
Let $f_h$ satisfy the conditions of Definition~\ref{def_B}. Then $\WFH(f)$ has no infinite points.   Let $\chi_1\in C_0^\infty$ be equal to $1$ in a neighborhood of $\supp f_h$ for all $h\ll1$, and $\chi_2$ be supported in the bounded $U\Supset\mathcal{B}$ and equal to $1$ near $\mathcal{B}$. Then \r{loc} is satisfied with $\psi=\chi_1(x) \chi_2(\xi)$. Indeed, by \r{defB}, for every $s$, 
\be{7}
(\Id - \chi_2(hD))f_h= O_{H_h^s}(h^\infty).
\ee
This implies the same estimate in the Schwartz space $\mathcal{S}$ as well. Apply $\chi_1(x)$ to get \r{loc}. 
Therefore, (a) $\Rightarrow$ (b). 

Next, assume (b).  
Let the compact set $\mathcal{B}\subset\R^n$ be such that for $\psi$ in \r{loc} we have $\psi(x,\xi)=0$ for all $x$ and $\xi\not\in\mathcal{B}$.  
With 
$U$ as in Definition~\ref{def_B}, we can apply a \sc\ Fourier multiplier $\Id-\chi_2(hD)$ with $\chi_2$ as above to get \r{7}. Therefore,  
\[
\int
\langle \xi\rangle^{2s}| (1-\chi_2(\xi))\mathcal{F}_hf_h(\xi)|^2\,\d\xi = O(h^\infty).
\]
Set $g(\xi) = \langle\xi\rangle^{s} (1-\chi_2(\xi))\mathcal{F}_hf_h(\xi)$. Using Sobolev embedding and the fact that $-\i\partial_\xi$ corresponds to $x/h$ via $\mathcal{F}_h$, we get $\|g\|_{L^\infty}= O(h^\infty)$. This proves \r{defB}. Therefore, (b) $\Rightarrow$ (a). 

Assume (c). Using \r{infWF} and a partition of unity, we get \r{defB}, i.e., (a) holds. On the other hand, (a) implies (c) directly. 
\end{proof}

Let $f_h$ be \sc ly band limited and let  $B$ be so that $\Sigma_h(f)\subset \{|\xi|<B\}$. Then 
$f_h=\chi_1(x) \chi_2(hD)f_h+O_\mathcal{S}(h^\infty)$ with $\chi_{1,2}$ as in the proof of Proposition~\ref{pr_bl}. We can assume $\supp\chi_2\subset \{|\xi|<B\}$.  Apply the operator $\langle hD\rangle^s$ to $f_h$, then on $\WFH(f)$, we get that $\langle hD\rangle^s\chi_1(x) \chi_2(D) $ has full symbol $\langle \xi\rangle^s$ up to the negligible class. Therefore, knowing $\|f_h\|_{L^2}$ for a \sc ly band limited $f_h$ allows us to control the \sc\ Sobolev norms of every order $s\ge0$ as well:
\be{Sob}
\|f_h\|_{H_h^s}\le  \langle B\rangle^s\|f_h\|_{L^2} + O_s(h^\infty). 
\ee

\subsection{$h$-\PDO s} We define the symbol class $S^{m,k}(\Omega)$, where $\Omega\subset \R^n$ is an open set, as the smooth functions $p(x,\xi)$ on $\R^{2n}$, depending also on $h$, satisfying the symbol estimates
\be{hpdo1}
|\partial_x^\alpha \partial_\xi^\beta p(x,\xi)|\le C_{K,\alpha,\beta}h^k\langle \xi\rangle^m
\ee
for $x$ in any compact set $K\subset\Omega$. The negligible class $S^{-\infty,\infty}$ is the intersection of all $S^{m,k}$. Given $p\in S^{m,k}(\Omega)$, we write $P=P_h=p(x,hD)$ with 
\be{hPDO}
Pf(x) = (2\pi h)^{-n}\iint e^{\i (x-y)\cdot\xi/h} p(x,\xi) f(y)\,\d y\, \d\xi,
\ee
where the integral has to be understood as an oscillatory one. If we stay with functions localized in phase space, the factor $\langle\xi\rangle^m$ is not needed and we can work with symbols compactly supported in $\xi$. Then the corresponding classes are denoted by $S^k(\Omega)$ and $k$  is called an order. One can always divide by $h^k$; so understanding zero order operators is enough.

\subsection{Classical \PDO s and \sc\ wave front sets} 
We begin with an informal  discussion about the relationship between classical and \sc\ \PDO s. Let us denote by $\xi$ the dual variable in the classical \PDO\ calculus, and by $\eta$ the dual variable  in the \sc\ case. Formally,  by \r{hPDO} (with $\xi$ replaced by $\eta$ there), we have $\eta=h\xi$ and after this substitution, we seem to get a classical \PDO. The problem is that classical symbols do not need to be smooth or even defined for $\xi$ in a compact set; say for $|\xi|\le C$ with some $C$. Then $\eta=h\xi$ maps this to $|\eta|\le Ch$. Semiclassical symbols however need to be defined and smooth for every $\eta$ in a neighborhood of the \sc\ wave front set of the function we want to study. We see that the zero section $\eta=0$ needs to be excluded. If we try to rectify this problem by multiplying a classical symbols $p(x,\xi)$ by $\chi(\xi)$ with  $\chi\in C_0^\infty$, $\chi=0$ near $\xi=0$ and $\chi=1$ for large $\xi$, we run into the problem that $\partial_\eta \chi(\eta/h)\sim h^{-1}$ and the symbol estimates \r{hpdo1} are not satisfied for $\chi(\eta/h) p(x,\eta/h)$ near $\eta=0$. 

This shows that the zero section needs to be treated separately. Even when that problem does not exists, classical ellipticity does not necessarily mean \sc\ one. For example, $-\Delta$ is a classical elliptic \PDO\ with symbol $|\xi|^2$ while $|\eta|^2/h^2$ is the \sc\ symbol of the same operator but $-\Delta$ is not \sc ly elliptic anymore (in $S^{2,-2}$).

\begin{proposition}
Let $K$ be a smoothing operator, and let $f_h\in\mathcal{E}'(\R^n)$ be tempered. Then $\WFH(Kf)\subset \R^n\times\{0\}$. 
\end{proposition}

\begin{proof}
Let $\phi,\psi\in C_0^\infty$ such that $\psi=1$ near some $\xi_0\not=0$ but $\psi=0$ near $0$. Then
\[
\begin{split}
\psi\mathcal{F}_h(\phi Kf)(\xi) &=  \psi(\xi)\int e^{-\i x\cdot\xi/h}\phi(x)  Kf(x)\,\d x \\
&= \psi(\xi)\iint e^{-\i x\cdot\xi/h}\phi(x)  K(x,y) f(y)\,\d x\,\d y \\
&=   \psi(\xi)\iint e^{-\i x\cdot\xi/h}  \phi(x)\left[(1-h^2\Delta_y)^m  K(x,y)\right] (1-h^2\Delta_y)^{-m} f(y)\,\d x\,\d y .
\end{split}
\]
For $m\gg1$, $(1-h^2\Delta)^{-m} f\in L^2$ with a norm $O(h^{-N})$ for some $N$. Fix one such $m$. We have $L^k e^{-\i x\cdot\xi/h} = e^{-\i x\cdot\xi/h}$ for every $k$ with  $L= \i h|\xi|^{-2} \xi\cdot \nabla_x$. Integrating by parts with $L^k$, we get $\psi\mathcal{F}_h(\phi Kf) = O(h^\infty)$. Therefore, every $(x,\xi)$ with $\xi\not=0$ is not in $\WFH(Kf)$. 
\end{proof}

\begin{theorem}\label{thm_PDO}
Let $P$ be a properly supported \PDO\  of order $m$. Then for every  $f_h$ localized in phase space,
\[
\WFH(Pf)\setminus 0\subset \WFH(f)\setminus 0.
\]
If $P$ is elliptic, then the inclusion above is an equality. 
\end{theorem}
\begin{proof}
By the proposition above, the property of the theorem is invariant under adding a smoothing operator as it should be. Let $p(x,\xi)$ be the symbol of $P$, so that $P=p(x,D)$ modulo a smoothing operator. Then $P$ is a formal $h$-\PDO\ with a symbol $p(x,\xi/h)$. The latter is a \sc\ symbol away from every neighborhood of  $\xi=0$. Indeed, for $x$ in every compact set $K$,
\[
|\partial_x^\alpha\partial_{\xi}^\beta p(x,\xi/h)|\le C_{\alpha, \beta,K}  h^{-|\beta|}  |\xi/h|^{m-|\beta|} = C_{\alpha, \beta,K} h^{-m}  |\xi|^{m-|\beta|}  \quad\text{for $|\xi/h|>1$}
\]
and in particular,
\[
|\partial_x^\alpha\partial_{\xi}^\beta p(x,\xi/h)|\le C_{\alpha, \beta,K}  h^{-m}  |\xi|^{m-|\beta|}\quad\text{for $|\xi|\ge\eps>0$}
\]
for every $\eps>0$ and $0<h\le \eps$. This shows that $p(x,\xi/h)$ is a \sc\ symbol of order $(m,-m)$ restricted to $|\xi|>\epsilon$. 
Note that the smallness requirement on $h$ depends on $\eps$. To complete the proof, we need to resolve this problem. 

We show next that for every fixed $\eps>0$, if $\WFH(f)\subset B_\eps(0)$, then $\WFH(Pf)\subset B_\eps(0)$ as well. This is a weaker version of what we want to prove but it is valid near $\xi=0$. 

Since $P$ is properly supported, with $\phi\in C_0^\infty$ as in the previous proof, we may assume that in $\phi Pf$, the function $f$ is supported in a fixed compact set. Then by a compactness argument, for every $\eps'>\eps$, we have $\mathcal{F}_h f(\xi)=O(h^\infty)$ for $|\xi|\ge\eps'$. Then 
\[
\begin{split}
\mathcal{F}_h(\phi Pf)(\eta)  &= (2\pi h)^{-n} \iiint e^{-\i x\cdot \eta/h +\i(x-y)\cdot\xi } \phi(x) a(x,\xi) f(y)\,\d y\, \d \xi\,\d x\\
 &= (2\pi h^2)^{-n} \iiint e^{-\i x\cdot \eta/h +\i(x-y)\cdot\xi /h} \phi(x) a(x,\xi/h) f(y)\,\d y\, \d \xi\,\d x\\
 & = (2\pi h)^{-n} \iint e^{\i x\cdot (\xi-\eta)/h} \phi(x) a(x,\xi/h) \mathcal{F}_hf(\xi)\, \d \xi\,\d x.
\end{split}
\]
The integration above can be restricted to $|\xi|\le\eps'$ with $\eps'>\eps$ fixed, and this will result in an $O(h^\infty)$ error.  
For the phase function $\i\Phi/h= \i x\cdot(\xi-\eta)/h$ we have  $\Phi_x= \xi-\eta$, i.e. the zeros are on the diagonal $\xi= \eta$ in the fiber variable. The following operator preserves $\exp(\i\Phi/h)$:
\[
L =- \i  \frac{\xi-\eta}{|\xi-\eta|^2}\cdot h\nabla_x .
\]
Since $|\xi|\le\eps'$, if we restrict $\eta$ to $|\eta|>\eps''$ with $\eps''>\eps'$, and integrate by parts, we get $O(h^\infty|\xi|^{-\infty} )$ above. This proves that $\WFH(Pf)$ is included in $\R^n\times B_{\eps''}(0)$. Since $\eps''>\eps$ can be taken as close to $\eps$ as we wish, this proves the claim. 

Now, using \HPDO\ cutoffs, we express $f$ as $f=f_1+f_2$, with $\WFH(f_1)$ included in $|\xi|\le\eps$ and $\WFH(f_2)$ included in $|\xi|\ge\eps/2$. By the claim above,  $\WFH(Pf_1)$ is   included in $|\xi|\le\eps$ as well. For $\WFH(Pf_2)$, by what we proved earlier, $\WFH(Pf_2)\subset \WFH(f_2)$. Therefore, $\WFH(f)\subset\WFH(f)\cup \R^n\times B_\eps(0)$ for every $\eps>0$ which proves that  $\WFH(f)\subset\WFH(f)\cup \R^n\times \{0\}$. In particular, this proves the first part of the theorem. 

To prove the second part, if $P$ is elliptic, there is a parametrix $Q$ of order $-m$ so that $QPf=f+Kf$, where $K$ is smoothing. Then we apply the first part of the proof. 
\end{proof}

For future reference, note that we also proved that for every $\eps>0$, every classical \PDO\ is  also a \sc\ one restricted to $f$ with $\WFH(f)$ not containing $\xi$ with $|\xi|\le\eps$. 

\subsection{Classical FIOs  and \sc\ wave front sets} 


\begin{theorem}\label{thm_FIO}
Let $A$ be an   FIO in the class $I^m(\R^{n_2},\R^{n_1},\Lambda)$, where $\Lambda\subset T^*(\R^{n_1}\times\R^{n_2})\setminus 0$  is a Lagrangian manifold and $m\in\R$.  Then for every 
$f_h$ localized in phase space, 
\be{FIO1}
\WFH(Af)\setminus 0\subset C\circ\WFH(f)\setminus 0,
\ee
where $C=\Lambda'$ is the canonical relation of $A$. 
\end{theorem}

\begin{proof}
The statement holds for $h$-FIOs, see, e.g, \cite{Guillemin-SC}. When $A$ is a classical FIO, $A$ can be written locally as
\[
Af(x) = (2\pi)^{-(n_1+n_2+2N)/4} \iint e^{\i\phi(x,y,\theta)}a(x,y,\theta)f(y)\,\d y\,\d\theta,
\]
modulo a smoothing operator, where $a$ is an amplitude of order $m+(n_1+n_2-2N)/4$ and $\phi$ is a non-degenerate phase function, see \cite[Chapter~25.1]{Hormander4} and $\theta\in\R^N$. As in the proof of Theorem~\ref{thm_PDO} above, we can express 
$Af$ as an oscillatory integral with a phase function $\phi(x,y,\theta)/h$ and an amplitude $a(x,y,\theta/h)$. The latter is a \sc\ amplitude for $|\theta|>\eps>0$ and $0<h<\eps$. The rest of the proof is as the proof of Theorem~\ref{thm_PDO} using the non-degeneracy of the phase.
\end{proof}

 As above,  we also proved that for every $\eps>0$, every classical FIO is  also a \sc\ one restricted to $f_h$ with $\WFH(f)$ not containing $\xi$ with $|\xi|\le\eps$. Finally, if $F$ has a left parametrix which is also an FIO, then \r{FIO1} is an equality. This happens, for example, if $C$ is locally a graph of a diffeomorphism and $A$ is elliptic. It also happens when $C$ satisfies the clean intersection condition, which is the case for the geodesic X-ray transform in dimensions $n\ge3$ with no conjugate points. 

Finally, we want to emphasize that while $\WFH(f)$ is just a set, the theorem (as typical  for such statements) gives us more than recovery of that set. If we know $Af=m$ up to an $O_\mathcal{S}(h^\infty)$ error, and $f_1$ and $f_2$ are two possibly different solutions, then $A(f_1-f_2)$ has an empty \sc\ wave front set; and if $A$ is, say elliptic and associated to a local diffeomorphism, then $\WFH(f_1-f_2)$ can only be contained in the zero section, i.e., we can recover $f$ microlocally, away from $\xi=0$; not just $\WFH(f)$.

\section{Sampling theorems}
\subsection{Sampling on a rectangular  grid} 
    We start with a version of the classical Nyquist--Shannon sampling theorem which allows for  oversampling. Below, $B>0$ is a fixed constant. 

\begin{theorem}\label{thm_sampling}
Let $f\in L^2(\R^n)$. Assume that $\supp \hat f\subset [-\Xi,\Xi]^n$ and let $0<s\le \pi/\Xi$. Let $\hat \chi\in L^\infty$ be supported in $[-1,1]^n$ and equal to $\pi^n$ on $B^{-1} .\supp\hat f$. If $ s\le  \pi/B$, then $f$ can be reconstructed by its samples $f(sk)$, $k\in\mathbf{Z}^n$   by
\be{2.1}
f(x) = \sum_{k\in\mathbf{Z}^n} f(sk) \chi\left(\frac{\pi}{s}(x-sk)\right),
\ee
and 
\be{2.2}
\|f\|^2 = s^n\sum_{k\in\mathbf{Z}^n} |f(sk)|^2.
\ee
\end{theorem}

\begin{proof}
Since $\hat f$ is  supported in $[-\Xi,\Xi]^n$, it is also supported in $[-\pi/s, \pi/s]^n$. Then we can take the $2\pi/s$  periodic extension $\hat f_\text{ext}$ of $\hat f$ in all variables and then the (inverse) Fourier series of that extension to get
\be{2.2b}
\hat f_\text{ext}(\xi) =   s^n \sum_k f(sk)e^{-\i s\xi\cdot k}.
\ee
Multiply this by $\pi^{-n}\hat \chi(s\xi/\pi)$ to get 
\be{2.2c}
\hat f(\xi) = (s/\pi)^n \hat \chi(s\xi/\pi)  \sum_k f(sk)e^{-\i s\xi\cdot k}.
\ee
Take the inverse Fourier Transform to get \r{2.1}. Equality \r{2.2} is just Parseval's equality applied to \r{2.2b}.
\end{proof}

Note that when $\hat\chi/\pi$ is the characteristic function of $[-1,1]$ in 1D, we get $\chi(x)= \sinc(x):=  \sin x/x$. In higher dimensions, we get a product of such functions. 
When $\supp \hat f\subset (-B,B)$, one can choose $\hat\chi\in C_0^\infty$, which makes the series \r{2.1} rapidly convergent. Even a piecewise linear $\hat \chi$ will increase the convergence rate: if for $n=1$ we choose $\hat\chi$ to have a trapezoidal graph by defining it as linear in $[\delta,1]$ for some $\delta\in (0,1)$ (and continuous everywhere), then 
\[
\chi(x) = \frac{\cos x-\cos(\delta x)}{(1-\delta)x^2} 
\]
which is $O_\delta(x^{-2})$ instead of just $O(|x|^{-1})$ as $\sinc$, with a constant getting large when $\delta$ is close to~$1$. 
Theorem~\ref{thm_sampling} says that the sampling rate should not exceed $\pi/B$, which is known as the Nyquist limit. The theorem can be extended to classes of non $L^2$ functions and then we need the sampling rate to be strictly below the Nyquist one. If $f=\sin(x)$ for example, $B=1$ is the sharpest band limit and a sampling rate of $\pi$ would yield zero values. On the other hand, that function is not in $L^2$. Sampling with a smaller step recovers that $f$ uniquely in the corresponding class.

\begin{remark}\label{rem_unit}
It is easy to see that the subspace $L_B^2$ of $L^2$ consisting of functions $f$ with $\supp \hat f\subset [-B,B]^n$ is a Hilbert space itself. If we consider the samples $f(sk)$ as elements of $\ell^2_s(\mathbf{Z}^n)$ with measure $s^n$ (i.e., sequences $a=\{a_k\}_{k\in\mathbf{Z}^n}$ with $\|a\|^2=s^n\sum_k|a_k|^2$), then the theorem implies that the map
\[
L_B^2\ni f\Longrightarrow f(sk)\in \ell^2_s(\mathbf{Z}^n)
\]
given by \r{2.1} with $s=\pi/B$ and $\chi=\sinc$, is unitary; i.e., not just an isometry but a bijective one because one can easily show that for every choice of the sampled values in $\ell^2_s(\mathbf{Z}^n)$ there is unique $f\in L_B^2$ with those values. 
Therefore, the set of the samples is not an overdetermined one for a stable recovery. Also, \r{2.1} is just an expansion of $f$ in an orthogonal basis. 

\end{remark}
 
There are several direct generalizations possible. First, one can have different band limits for each component: $|\xi_j|\le B_j$ on $\supp\hat f$. Then we need  $s_jB_j/\pi\le\delta<1$. One can have frequency support in non-symmetric intervals but since we are interested mostly in real valued functions, we keep them symmetric. 

The next observation is that the set containing  $\supp\hat f$ does not have to be a box. Of course, if it is bounded, it is always included in some box but more efficient sampling can be obtained if we know that the group generated by the shifts $\xi_j\mapsto \xi_j+2B_j$, $j=1,\dots,n$ maps $\supp \hat f$ into mutually disjoint sets. Then we take the periodic extension of $\hat f$ w.r.t.\ that group. The requirement for $\hat\chi$ then is to have disjoint images under those shifts and to be equal to $1$ on $\supp\hat f$; then the step from \r{2.2b} to \r{2.2c} works in the same way. Note that $\hat\chi$  may not be  contained in $\prod [-B_j,B_j]$. 

Finally, one may apply  general non-degenerate linear transformation as in \cite{PetersenM}. It is well known that the change $x=Wy$ with $W$ an invertible matrix triggers the change $\xi=(W^*)^{-1}\eta$ of the dual variables. Let $f$  be  band limited in $\mathcal{B}$. 
Assume that $W$ is such that the images of $\mathcal{B}$ under the translations $ \xi\mapsto \xi+2\pi  (W^*)^{-1}k $, $k\in\mathbf{Z}^n$ are mutually disjoint. Then in the $y$ variables, the shifts of $\supp \hat f$ by the group $\eta\mapsto 2\pi\mathbf{Z}^n$ are mutually disjoint (i.e., for $g(y):=f(Wy)$,  we have that for $\hat g(\eta)$). 
 Then $g$ is  band limited in a box with a half-side  $B=\pi$ and then for every $s\in (0,1]$,  $f$ is  stably  determined by its samples $f(sWk)$, $k\in\mathbf{Z}^n$. Moreover, the  reconstruction formula  \r{2.1}  still holds with the requirement that $\chi=1$ on $\supp\hat f$ and $\supp\chi$ has disjoint images under the translations by that group; there is an additional $|\det W|$ factor (see also Theorem~\ref{thm_sc}) coming from the change of the variables.

We present next a \sc\ version of the sampling theorem. It can be considered as an approximate  rescaled version of the classical theorem when the conditions are approximately held,  with an error estimate. We present a version with a uniform estimate of the error and in the corollary below, we formulate a corollary without that uniformity but which requires $f_h$ to be \sc ly band limited (only), i.e., it applies to a single $f_h$. 



\begin{theorem}\label{thm_sc}
Assume that $\Omega\subset \R^n$, 
$\bar{\mathcal{B}}_0\subset \mathcal{B}\subset\R^n$ with $\Omega$,  $\mathcal{B}_0$, $\mathcal{B}$ open and bounded. Let $f_h\in C_0^\infty(\Omega)$ satisfy 
\be{loc1}
\|(\Id - \psi(x,hD))f_h\|_{H_h^m}=O_m(h^\infty)\|f_h\|,\quad \forall m\gg0
\ee 
with some $m$,  $\psi\in C_0^\infty(\R^{2n})$ so that $\psi(x,\xi)=1$ for $\xi \in \mathcal{B}_0$ and  $\psi(x,\xi)=0$ for $\xi\not\in\mathcal{B}$. Let $\hat \chi\in C_0^\infty(\R^n)$ be so that $\supp\hat\chi\in\mathcal{B}$ and $\hat \chi(\xi)\psi(x,\xi)=\psi(x,\xi)$. 

Assume that $W$ is an invertible matrix so that the images of $\mathcal{B}$ under the translations  $ \xi\mapsto \xi+2\pi  (W^*)^{-1}k $, $k\in\mathbf{Z}^n$, are mutually disjoint. 
Then  for every $s\in (0,1)$, 
\be{2.3'}
f_h(x) = |\det W| \sum_{k\in \mathbf{Z}^n} f_h(shWk)  \chi\left(\frac{\pi}{sh}(x-shWk)\right) + O_{H^s}(h^\infty)\|f\|_{L^2},
\ee
and 
\be{2.4'}
\|f_h\|_{L^2}^2 = |\det W|(sh)^n\sum_{k\in \mathbf{Z}^n} |f_h(shWk)|^2 + O(h^\infty)\|f\|_{L^2}^2.
\ee
\end{theorem}

\begin{proof} As in the remark above, we can make the change of variables $x=Wy$; then the dual variables $\eta$ is related to $\xi$ by $\eta=W^*\xi$. In the new variables, the images of $\mathcal B$ under the translations $\eta\to\eta+2\pi k$ do not intersect. Therefore, it is enough to consider the case $W=\Id$. 

 Set 
\be{2.5'}
g_h(x):= \mathcal{F}_h^{-1}\chi \mathcal{F}_h f_h  .
\ee
Since 
 $\mathcal{F}_hg_h$  is the classical Fourier transform of the rescaled $h^ng_h(hx)$,  by the previous theorem and the remark after it, with $B=\pi$, 
\[
h^ng_h(hx)=    \sum_k h^n g_h(shk) \chi(\pi(x -sk  )/s).
\]
Replace $hx$ by $x$ to get \r{2.3'}  and \r{2.4'} for $g_h$ without the error terms, i.e.,
\be{2.5''}
g_h(x) = \sum_{k\in \mathbf{Z}^n} g_h(shk)  \chi\left(\frac{\pi}{sh}(x-shk)\right), \quad \|g_h\|_{L^2}^2 = (sh)^n\sum_{k\in \mathbf{Z}^n} |g_h(shk)|^2 . 
\ee

To estimate the error, write
\[
\begin{split}
r_n:&=f_h-g_h = \mathcal{F}_h^{-1}(1-\chi) \mathcal{F}_h f_h  = (\Id-\chi(hD))f_h\\
& = 
 (\Id-\chi(hD))\left( \psi(x,D)f_h+\tilde r_h\right),
\end{split}
\]
where $\tilde r_h$ satisfies \r{loc1}. Therefore, $r_n$ satisfies the same error estimate. 
Using Sobolev embedding, we get $r_n(shk)=O(h^\infty)\|f\|$ for $shk$ in any fixed neighborhood of $\bar\Omega$. Outside of it, $r_h=g_h$ is in the Schwartz class by \r{2.5'} with every seminorm bounded by $C_Nh^N\|f\|$, $\forall N$. This implies that replacing $g_n(shk)$ by $f_n(shk)$ in the interpolation formula in \r{2.5''} results in an $O(h^\infty)$ error in every $H^s_h$. We already established such an error estimate if we replace $g_h$ by $f_h$ on the left. The Parseval's identity  \r{2.4'} follows from this as well. 

This completes the proof. 
\end{proof}


In particular, when we are given a single \sc ly band limited $f_h$, we have the corollary below. Note that estimates \r{WFH},  \r{infWF}, \r{loc}, and \r{defB} are not formulated uniformly in $f_h$; and on the other hand, estimate \r{WFH} is the standard definition of the \sc\ wave front set.  Theorem~\ref{thm_sc} has uniform estimates of the error but requires the stronger condition \r{loc1}. Also, one can have the \sc\ band limited assumptions as in the corollary below but still the more general sampling geometry as in  Theorem~\ref{thm_sc}.

\begin{corollary}\label{cor_sc}
Let $f_h$ be \sc ly band limited with $\Sigma_h(f) \subset  \prod(-B_j,\B_j)$. Let $\hat \chi_j \in C_0^\infty(\R)$ be supported in $(-1,1)^n$ and $\hat \chi_j(\xi_j/B_j)=\pi$ for $\xi\in \Sigma_h(f) $. If $0<s_j \le \pi/B_j$, then 
\be{2.3}
f_h(x) = \sum_{k\in \mathbf{Z}^n} f_h(s_1 hk_1,\dots, s_nhk_n) \prod_j \chi_j\left(\frac{\pi}{s_j h}(x-s_jhk)\right) + O_{\mathcal{S}}(h^\infty), 
\ee
and 
\be{2.4}
\|f_h\|^2 = (sh)^n\sum_{k\in \mathbf{Z}^n} |f_h(shk)|^2 + O(h^\infty). 
\ee
\end{corollary}

\begin{proof}
Take $\Omega\supset\supp f$, $W=\text{diag}(\pi/B_1,\dots,\pi/B_n)$. Then $\det W =\pi^n/(B_1\dots B_n)$. 
If $\delta<1$ is close enough to $1$, then the conditions of Theorem~\ref{thm_sc} are satisfied and we can take $\hat \chi$ there (depending on $\xi\in\R^n$) to be the product $\hat \chi_0 (\xi_1/B_1)\dots\hat \chi_0 (\xi_1/B_n)$, where $\chi_0$ (depending on an 1D variable) is as in the corollary but related to $B_j=1$.  Set $\hat \chi_j(\xi)=\hat \chi_0(\xi/B_j)$; then $\chi_j(x) = B_j\chi_0(B_jx)$. On the other hand, having $\chi_j$, we can compute $\chi_0$. Then Theorem~\ref{thm_sc} implies the corollary. 
\end{proof}

In other words, if the sampling rate $sh$ is smaller  than $\pi h/\Xi$, we have an accurate recovery. Note that the sampling rate $s$ is rescaled by $h$ compared to Theorem~\ref{thm_sampling}.  We call $s$ a relative sampling rate.

\begin{remark}\label{remark_W}
Let $\supp f\in \Omega\Subset\R^n$ be $B$-band limited as in the corollary. Then  the smallest subset in the phase space $T^*\Omega$ containing $\WFH(f)$ for all such $f$'s is $\bar\Omega\times [-B,B]^n$. Its volume with respect to the volume form $\d x\,\d\xi$ is
\[
\Vol\left(\bar\Omega\times [-B,B]^n\right) = (2B)^n\Vol(\Omega). 
\]
 The number of samples we need for $f$ is bounded from below by 
\be{2.7}
\text{ $\#$ of samples}=    \frac{\Vol(\Omega)}{(\pi h/B)^n}= (2\pi h)^{-n}\Vol\left( \bar\Omega\times [-B,B]^n\right)
\ee
up to an $o(1)$ relative error as it follows by comparing the volume of $\Omega$ with the number of the lattice points we can fit in it. The factor $(2\pi h)^{-n}$ is sometimes multiplied by $\d\xi$ to define a natural rescaled measure $(2\pi h)^{-n}\d\xi$ in the $\xi$ space. Therefore, we see that the asymptotic number of samples needed to resample such $f$'s has a sharp lower bound equal to the phase volume occupied by $f$ in the phase space w.r.t.\ to that \sc\ measure. Note that this is sharp when applied  to $f$'s  with $\WFH(f)$ contained in the product of $\Omega$ an a box in the $\xi$ variable. This is a version of the lower bound in \cite{Landau-sampling} in  our setting, which we prove in Theorem~\ref{thm_Landau}. 
\end{remark}

\subsection{Non-linear transformations and non-uniform sampling}\label{sec_non-un}
 Non-uniform sampling in dimensions $n\ge1$ is not well understood unless we sample on Cartesian products of non-uniform 1D grids. We  mention here some easy to obtain but not very far reaching extensions. If $x=\phi(y)$ is a diffeomorphism, one may try to obtain sampling theorems for $f(x)$ by applying the results above to $g(y):=f(\phi(y))$ if the latter happens to be band-limited. Then the sampling points $y_k$, $k\in\mathbf{Z}^n$ will give us sampling points $x_k=\phi(y_k)$. This is what we did above with $\phi(y)=Wy$ linear. Then we first choose $V$ so that the translates of $\supp \hat f$ (or $\Sigma_h(f)$) under the group $\eta\to \eta+V\mathbf{Z}^n$ are mutually disjoint and find $W$ from the equation $V=2\pi(W^*)^{-1}$. Let us consider the \sc\ sampling theorem now. The analog of this for non-linear transformations is the well known property that $\WFH(f)$ transforms as a set of covectors, i.e., $\xi=((\d\phi)^*)^{-1}\eta$.  Now, if we want to prescribe $\d\phi(y)$ pointwise (at $y=\phi^{-1}(x)$ for every $x$) based on an a priori knowledge of $\WFH(f)$, we run into the problem that the equation $\d\phi=\Phi$ with $\Phi$ given is not solvable unless $\d\Phi=0$ (i.e., the matrix-valued map $\Phi$ is curl-free). On the other hand, we can take a partition of unity and on each covering open set, we can take a linear transformation giving is better sampling locally. In Section~\ref{sec_R_PG}, we prove a lower bound for the number of sampling points which is sharp at least in the situation of Theorem~\ref{thm_sc}. 

\subsection{Aliasing.} \label{sec_alias}
Aliasing is a well known phenomenon in classical sampling theory when there Nyquist limit condition is not satisfied (i.e., we have undersampling). For simplicity, we will recall the basic notions when the sampling is done on a rectangular grid but the more general periodic case in Theorem~\ref{thm_sc} can be handled similarly.

Assume that we sample each $x^j$-th variable of $f(x)$ with a steps $sh$. Different steps $s_j$ can be handled easily with a diagonal linear transformation.  We do not assume that this steps satisfy the Nyquist condition in Corollary~\ref{thm_sc}. 
Assume also that we use formula \r{2.3}.  
Following the proof of Theorem~\ref{thm_sampling}, \r{2.2b} 
still holds but the periodic extension $\hat f_\text{ext}$ consists of a superposition of periodic shift of $\hat f_h$ w.r.t.\ to the group $(2\pi/s)\mathbf{Z}^n$:
\be{A1}
\hat f_{\text{ext}}(\xi) = \sum_{k\in\mathbf{Z}^n}\hat f (\xi+(2\pi/s)k).
\ee
Now they can overlap and the restriction of $\hat  f_{\text{ext}}$ to $[-\pi/s,\pi/s]^n$ is not $\hat f$ anymore.  Then \r{2.2c} needs to be corrected by replacing the l.h.s.\ by $\hat \chi(s\xi/\pi) \hat  f_{\text{ext}}(\xi)$. Then the proof of Theorem~\ref{thm_sc} shows that the r.h.s.\ of \r{2.3} approximates not $f_h$ but of 
\be{A2}
G f_h := \mathcal{F}_h^{-1} \hat\chi(s\cdot/\pi)(\mathcal{F}_h f_h)_\text{ext}. 
\ee
In particular, high (outside $[-\pi/s,\pi/s]^n$) frequencies $\xi$ will be shifted by $(2\pi/s)k$ for some $k\in\mathbf{Z}^n$ so they land in that box and their amplitudes will be added to the ones with that actual frequency  there. This is known as ``folding'' of frequencies. Note also that for $f_h$ real valued, $\WFH(f)$ and $\Sigma_h(f)$ are even, i.e., symmetric under the transform $\xi\mapsto-\xi$. 

In fact, $ G=\sum_{k\in\mathbf{Z}} G_k$, where each $G_k$ is an h-FIO with a canonical relation given by the shifts 
\be{sh}
S_k:(x,\xi) \longmapsto    (x,\xi+{2\pi} k/s).
\ee
Indeed, we have
\be{A9}
\begin{split}
G_kf(x) &= (2\pi h)^{-n} \int e^{\i(x-y)\cdot\xi/h - 2\pi \i k\cdot y/sh} \hat\chi(s\xi/\pi) f(y)\, \d y\,\d\xi\\
 &= e^{- 2\pi \i k\cdot x/sh}   \hat\chi(shD/\pi +2 k)f.
\end{split}
\ee
This FIO view of aliasing will be very helpful later. 

 In applications, other reconstructions are used, for example splines. Without a proof, we will mention that the aliasing artifacts are similar then. 

As an example, we plot the function $f_h$ consisting of a sum of the real parts of a sum of two coherent states, see \r{coh}, with $h=0.01$ and $h=0.04$, respectively and unit $\xi_0$'s, on the rectangle $[-1,1]^2$. If we take $h=0.1$ as the small parameter, then the higher frequency state requires a sampling rate $sh$ with $s<\pi$, therefore, it needs more than $2/(\pi h)\approx 64$ sampling points in   each variable. The lower frequency set requires about $1/4$ of that. In Figure~\ref{fig_aliasing}, we sample on $81\times 81$ point grid fist, and then on a $41\times 41$ one next. In the first cases, both patterns are oversampled, while in the second one, only one is, and the other one is undersampled. The reconstructed images, using the {\tt lanczos3}  interpolation in MATLAB (instead of  \r{2.3}) are shown; the aliasing of the smaller pattern changes  the direction and the magnitude of the frequency. In the third and the fourth plots, we show the absolute values of the Fourier transforms of both images: the oversampled one is plotted in its Nyquist box while the white box is the Nyquist box of the undersampled one next to it. The absolute value Fourier transform of the undersampled image is plotted last. One can see that the frequencies outside the Nyquist box in the previous case have shifted by $(0,-2B)$ and $(0,2B)$, where $2B$ is the side of the wide box (with the so chosen $h$, we have $B=1$). 

\begin{figure}[h!] 
  \centering
	\includegraphics[trim = 0 0 0 0, clip,height=.15\textheight]{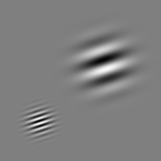}\ \ \ \ \ \ 
  \includegraphics[trim = 0 0 0 0, clip, height=.15\textheight]{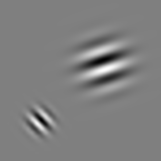}\ \ \ \ \ \ 
    \includegraphics[trim = 0 0 0 0, clip, height=.12\textheight]{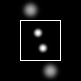}     \ \ \ \ \ 
    \includegraphics[trim = 0 0 0 0, clip, height=.12\textheight]{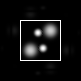}
\caption{\small From left to right: (a) the original $f_\text{o}$; (b) $f_\textrm{u}$ with a lower sampling rate:  the lower left pattern is undersampled, the larger one is oversampled; (c) $|\hat f_\text{o}|$ with the white box representing the Nyquist limit in the next plot; (d) $|\hat f_\text{u}|$, the white box is the Nyquist limit. }
\label{fig_aliasing}
\end{figure}

\section{Sampling classical FIOs} We are ready to formulate the main results of this paper. We label them by (i), (ii), (iii) and (iv) as in the Introduction. 

\subsection{(i) Sampling $Af$.} Let $A$ be a classical FIO as in Theorem~\ref{thm_FIO}. If we know a priori that $f$ is \sc ly band limited, then so is $Af$ and we can find a sharp upper bound on its band limit. This determines a sharp sampling rate for $Af$. 

\begin{theorem}\label{thm_main}
Let $A$ be a classical FIO as in Theorem~\ref{thm_FIO} with canonical relation $C$. Let $f_h$ be \sc ly band limited. 
Assume that  $C\circ \WFH(f)\subset\Omega\times\mathcal{B} $ with $\Omega$ a bounded domain and $\mathcal{B}$ bounded. If $W$ is as in Theorem~\ref{thm_sc}, then $Af$ can be reconstructed in $\Omega$ up to an $O(h^\infty)$ error from its values on the lattice $shWk$, $k\in\mathbf{Z}$, $0<s<1$  in the sense of \r{2.3'}, \r{2.4'}.  
\end{theorem}

If we do not require uniformity of the error in $Af$, one can take $s=1$. 

\subsection{(ii) Resolution limit of $f$ given the sampling rate of $Af$.} 
Assume now that we sample $Af$ at a rate $s_jh$ in the $x^j$-th variable with some fixed $s_j>0$; or on a more general periodic lattice.  What resolution limit does this impose on $f$?

By  Corollary~\ref{thm_sc}, to avoid aliasing, $\Sigma_h(Af)$ should be in the box $[-B_j, B_j]$ with $B_j>\pi/s_j$ (we assume that we deal with real valued functions, and therefore always work in frequency sets symmetric about the origin). Then if 
\be{C1}
\pi_2\circ C\circ\WFH(f)  \subset     \prod_j (-B_j,B_j),
\ee
there is no loss (up to $O(h^\infty$) when the data $Af$ has been sampled. Note that the  canonical relation $C$ does not need to be an 1-to-1 map. If $C$ is a local diffeomorphism and if $A$ is elliptic, then \r{C1} is sharp in the sense that if it is violated, $f$ cannot be recovered up to $O(h^\infty)$, i.e., there will be aliasing. This happens for the 2D Radon transform, for example. 

We write \r{C1} in a different way.  Then  \r{C1} is equivalent to
\be{C2}
\WFH(f)\subset  C^{-1}\circ \Big(\R^m\times \prod_j (-B_j,B_j)\Big),
\ee
assuming that $A$ takes values in $\R^m$. Similarly to $C$, the inverse canonical relation $C^{-1}$ does not need to be an 1-to-1 map.

Relation \r{C2}  gives easily a sharp limit (sharp when $A$ is elliptic, associated to a local diffeomorphism) on $\Sigma_p(f)$ guaranteeing no aliasing. It actually says something more: the resolution limit on $f$ is microlocal in nature, i.e., it may depend on the location and on the direction. We illustrate this below with the Radon transform. 

\subsection{(iii) Aliasing artifacts} \label{sec_alias_FIO}
When \r{C2} is violated and when $A$ is elliptic, associated to a local diffeomorphism, for example, there will be aliasing of $Af$. To understand how this affects $f$, if we use back-projection, i.e., a parametrix $A^{-1}$, 
we recall that the aliasing can be interpreted as an h-FIO associated to shifts of the dual variable, see \r{sh} and \r{A9}. Then the inversion would be $A^{-1}G_kA$; and by Egorov's theorem,  that is an h-FIO with a canonical relation being $C^{-1}\circ S_k\circ C$ acting on $(x,\xi)\in\supp\hat \chi(s\cdot/\pi+2k) $. 
We will not formulate a formal theorem of this type; instead we will illustrate it in the examples in the next sections. The classical aliasing described in Section~\ref{sec_alias} creates artifacts at the same location but with shifted frequencies. The artifacts here however could move to different locations, as it happens for the Radon transform, for example. 

Note that $A^{-1}$ has  canonical relation $C^{-1}$ but the latter acts on the image of $T^*\Omega\setminus 0$ if a priori $\supp f\subset \Omega$ and we restrict the reconstruction there. Some of the shifted singularities of $Af$ may fall outside that domain of $C^{-1}$ and they will create no singularities in the reconstruction. Therefore, we may have shifts in the reconstructed $f$, full or partial cancellation (and interference patterns as a result), and removal of singularities even if their images are still present in $Af$. 

\subsection{(iv) Locally averaged sampling} 
We study now what happens when the measurements $Af$ are locally smoothened either to avoid aliasing or for some practical reasons. One way to model this is to assume that we are given samples of $\phi_h*Af$, where $\phi_h= h^{-m}\phi(\cdot/h)$, where $m$ is the dimension of the data space, and $\phi$ is smooth with $\int\phi=1$. If $\hat\phi(\eta)$ is approximately supported in the ball $|\eta|\le B'$, how to sample  $\phi_h*Af$ and what does this tell us for $f$? The answer to the first question is given by the sampling theorems --- we need to sample at rate smaller than $\pi h/B'$ assuming that $Af$ a priori has even higher frequencies than $B'$. The   second question is more interesting. As explained in the Introduction, we have that  $\phi_h*$ is a \HPDO\ and if $A$ is associated to a canonical map, then  and by Egorov's theorem, $\phi_h*Af =AP_hf+O(h^\infty)f $, where $P_h$ is a \HPDO\ of order zero with principal symbol $\hat\phi\circ C$. If $A$ is well posed, we can recover $P_hf$ up to small error which is a certain regularized version of $f$. 

We can make this more general. Assume that the convolution kernel can depend on the sampling point. We model that by assuming that we are sampling $Q_hAf$, where $Q_h$ is an \HPDO\ of order zero with essential support of its symbol contained in $(y,\eta)$ for which $|\eta|\le B'$. This hides the implicit assumption that the convolution kernels cannot change rapidly when we make $h$ smaller and smaller because the symbol $q(y,\eta)$ of $Q_h$ must satisfy the symbol estimate \r{hpdo1}, and in particular, $\partial_xp$ cannot increase with $h^{-1}$. Then it is not hard to see that at every sampling point $y=y_j$, $Q_hAf(y_j)$ is the Fourier multiplier with $q(y_j,\eta)$ restricted to the point $y=y_j$. Therefore, each measurement is really a convolution. An application of the \sc\  Egorov theorem \cite{Guillemin-SC} combined with the remark following Theorem~\ref{thm_FIO} yields the following. 

\begin{proposition}\label{prop_av}
Let $f$ be \sc ly band limited. For $\eps>0$, let $f=f_1+f_2$ with $\Sigma_h(f_1)\subset \{|\xi|\le\eps\}$ and $\Sigma_h(f_2)\subset\{|\xi|\ge 2\eps  \}$. Let $F$ be a classical FIO associated with a  diffeomorphism $C$ and let $P_h$ be a \HPDO. Then $Q_hFf_1 = FP_hf_1+O(h^\infty)f_1 $ with $P_h$ a \HPDO\ with principal symbol $q_0=p_0\circ C$, where $q_0$ is the principal symbol of $Q_h$. Moreover, the full  symbol $p$ of $P_h$ is supported in $C^{-1}(supp(q))$, where $q$ is the full symbol of $Q_h$. 
\end{proposition}

Therefore, having locally averaged  data $Q_hAf$ instead of $Af$, with $A$ elliptic, allows us to reconstruct the smoothened $P_hf$ (plus a function with much lower frequencies) instead of $f$. If we want to choose $P_h$ first, for example to be a specific convolution, we can find the operator $Q_h$ which applied to the data $Af$ results in the desired regularization. Finally, in applications, $Q_hA(y_j)$ may not be realized as convolutions with compact supports of their Fourier transforms but if those supports are approximately compact with some error estimates, we can apply the asymptotic isometry property \r{2.4} to estimate the resulting error.

\section{Non-uniform sampling: A lower bound of the sampling rate}\label{sec_NU}
  Assume we want to sample a \sc ly limited $f_h$ on a non-uniform grid, see also Section~\ref{sec_non-un}. One reason to do that would be to reduce the number of the sampling points if the shape of $\WFH(f_h)$ allows for this, as in the Radon transform examples (where we sample $\mathcal{R}f$). 
We prove a theorem similar to one of the results of Landau \cite{Landau-sampling} in the classical case. In case of non-uniform sampling, we establish a lower bound of the sampling rate of $f$ with $\WFH(f)$ contained in a fixed compact set, equal to the phase volume of the interior of the latter. In most applications, $\partial\mathcal{K}$ would have measure zero, and then $\Vol(\mathcal{K}^\text{\rm int})= \Vol(\mathcal{K})$. The theorem below is different that the corresponding results in  \cite{Landau-sampling} in the following way. Aside from being semiclassical, our bound is in terms of the volume of $\WFH(f)$, i.e., it is microlocalized, rather than being  $(2\pi)^{-n}\Vol(\Omega)\Vol(\mathcal{B})$ for the number of points in  every $\Omega$ when $\supp \hat f\subset \mathcal B$. In other words, we express the bound in terms of the volume of $\WFH(f_h)$ instead of the volume of the minimal bounding product $\Omega\times\mathcal B$. 

\begin{theorem}\label{thm_Landau}  
Let $\{x_j(h)\}_{j=0}^{N(h)}$ be a set of points in $\R^n$. 
Let $\mathcal{K}\subset T^*\R^n$ be a compact set. 
If 
\be{L1}
\|f_h\| \le C \sum_{j=1}^{N(h)} |f_h(x_j(h))|^2+O(h^\infty)
\ee
for every \sc ly band limited $f_h$ with $\WFH(f)\subset \mathcal{K}$, then 
\be{L2}
N(h)\ge (2\pi h)^{-n}\Vol(\mathcal{K}^\text{\rm int})(1-o(h)),
\ee
where $\Vol(\mathcal{K}^\text{\rm int})=\int_{\mathcal{K}^\text{\rm int}}\d x\,\d \xi$ is the measure of the interior $\mathcal{K}^\text{\rm int}$ of $\mathcal{K}\subset T^*\R^n$. 
\end{theorem}

\begin{proof} 
By the properties of the Lebesgue measure, given $\delta>0$, we can find a closed (and necessarily compact in this case) set $F_\delta\subset \mathcal{K}^\text{\rm int}$  so that $\Vol( \mathcal{K}^\text{\rm int} \setminus F_\delta)<\delta$. 
 Let $0\le p(x,\xi)\le 1$ be a real valued smooth function supported in   $\mathcal{K}^\text{\rm int}$  and equal to $1$ on $F_\delta$.  Let $P_h=p^\text{w}(x,hD)$  be the Weyl quantization of $p$. Then $P_h$ is self-adjoint and compact. 

Let $E_h$ be the eigenspace of $P_h$ spanned by the eigenfunctions corresponding to the eigenvalues in 
$[1/2,2]$. The upper bound $2$ can be replaced by any number greater than $1$ since the eigenvalues of $P_h$ cannot exceed $1+O(h)$. For simplicity (not an essential assumption for the proof), assume that $1/2$ is a non-critical value for $p$. 
Every $f_h\in E_h$  satisfies $\WFH(f_h)\subset \mathcal{K}$. Indeed, for every unit  eigenfunction $\phi_h$ of $P_h$ with an eigenvalue  $\lambda_h\in [1/2,2]$, and for every $q\in S^{0,0}$ with $\supp q\cap \mathcal{K}=\emptyset $ we have $q(x,hD)\phi_h =(1/\lambda_h)q(x,hD) p(x,hD)\phi_h=O(h^\infty)$, and this yields the same conclusion for every tempered $f_h\in E_h$.

 By \cite{DimassiSj_book}, we have the following Weyl asymptotic
\be{dim}
\dim E_h= (2\pi h)^{-n}\Vol\left(1/2\le p\le 2\right)(1+O(h)). 
\ee
We will show that $N(h)\ge\dim E_h$ for $0<h\ll1$. If this is not true, we would have $N(h)<\dim E_h$ for a sequence $h=h_j\to0$. Then for every $h\in\{h_j\}$, we can find a unit $f_h\in E_h$ so that all the samples of $f_h$ vanish for such $h$. As we showed above, $\WFH(f_h)\subset \mathcal{K}$. This contradicts \r{L1} however. 

Therefore, by \r{dim},
\[
\Vol(F_\delta)< \Vol\left(1/2\le p\le 2\right)\le \liminf_{h\to0} (2\pi h)^nN(h) . 
\]
Take the limit $\delta\to0$ to complete the proof.  
\end{proof}

\begin{remark}
The proof holds if we replace the error term in \r{L1} by $o(h)\|f\|$.
\end{remark}

\begin{remark}\label{rem_L1}
Existence and a characterization of  optimal sampling sets where \r{L2} would be an equality is a harder problem which we do not study here. 
Estimate \r{L2} is sharp in case of uniform sampling described in Corollary~\ref{cor_sc} at least, see Remark~\ref{remark_W} which generalizes easily to different band limits $B_j$ for each component of $x$ as in Corollary~\ref{cor_sc}. 
It is straightforward to show that \r{L2} is also sharp in the case of more general uniform sampling described in Theorem~\ref{thm_sc}. 
\end{remark}

\begin{remark}\label{rem_L2}
The statement of the theorem is preserved under (non-linear) diffeomorphic transformations because $\WFH$ and its phase volume are invariant. 
If for some $\mathcal{K}$, we can choose a non-linear transformation which would fit the problem in the situation handled by Theorem~\ref{thm_sc}, we can construct a sampling set with the optimal number of sampling points by transforming the periodic lattice by that transformation. Doing this piece-wise, as suggested in  Section~\ref{sec_non-un}, would provide a smaller sampling set. We show how this can be done in our Radon transform examples.
\end{remark}

\begin{corollary}\label{cor_num}
Let $A$ be a classical FIO of order $m$ associated with a diffeomorphic canonical relation $C$. Then the minimal asymptotic number of points (up to an $1+o(h))$ relative error) to sample $Af_h$ guaranteed by Theorem~\ref{thm_Landau}   does not exceed the number of points needed to sample $f_h$; and if $A$ is elliptic, it is the same. 
\end{corollary}

The proof follows from the fact that $C$ is symplectic and in particular it preserves the phase volume; and from Theorem~\ref{thm_FIO}. 

In the examples we consider, $A$ happened to be an 1-to-2 diffeomorphism, and each branch is elliptic. Then we can  apply the corollary to each ``half'' of $C$. Then the number of points to sample $Af$ stably would be twice that for $f$; but the number of points to recover $f$ stably up to a function with $\WFH$ in a small neighborhood of the zero section is half of that.

In the remainder of the paper, we present a few applications.

\section{The X-ray/Radon transform in the plane in the parallel geometry} \label{sec_R_PG}
We present the first example: the X-ray/Radon transform $\mathcal{R}$ in the plane in the so-called parallel geometry parameterization. The analysis of this transform in this and in the fan-beam representation  can and has been done with traditional tools \cite{Bracewell, Rattey-sampling, Natterer-sampling1993}, also \cite[Ch.~III]{Natterer-book}, when the weight is constant since the symmetry allows us to relate that transform to the Fourier transform of $f$ by the Fourier Slice Theorem.  
The sampling of $\mathcal{R}f$ however requires estimates of the Fourier transform of $\mathcal{R}f$, which is done in those papers in a non-rigorous way by Bessel functions expansions and their asymptotics.

We go a bit deeper than that even when the weight is constant and we treat variable weights as well. The main purpose of this section is to demonstrate the general theory on a well studied transform, where one can write explicit formulas; and sampling analysis has been done (for constant weights), so we can compare the results. 

The numerical simulations in this and in the next section have been done in Matlab. The phantoms $f$ are defined by formulas and sampled first on a very fine grid. Then we compute their Radon transforms $\mathcal{R}  f$ numerically. To simulate coarser sampling of $\mathcal{R}  f$, we sample the so computed $\mathcal{R}  f$. To simulate inversion with coarsely sampled $\mathcal{R}  f$, we upsample the downsampled data to the original grid to simulate a function of continuous variables. Instead of using the Whittaker-Shannon interpolation formula \r{2.3}, we use the {\tt lanczos3}  interpolation which is a truncated version of the latter. Then we perform the inversion on that finer grid. Note that our goal is not to reduce the computational grid at this point, it rather is to show the amount of data and the artifacts contained in data sampled in a certain way. We compute the Fourier transforms of $f$ and $\mathcal{R}f$ using the discrete Fourier transform command in Matlab. Since we work with $f$ vanishing near the boundary of the square $[-1,1]^2$, and $\mathcal{R}f$ is vanishing in the $p$ variable near $p=\pm 1$ (for such $f$'s) and is periodic in its angular variables, the  discrete Fourier transform, which in fact is a transform on a torus, gives no  artifacts. 


\subsection{$\mathcal{R}_\kappa$ as an FIO} 
  Let $\mathcal{R}_\kappa$ be the weighted Radon transform in the plane
\be{R1}
\mathcal{R}_\kappa f(\omega,p) = \int_{x\cdot \omega=p}\kappa(x,\omega)  f(x)\,\d\ell,
\ee
where $\kappa$ is a smooth weight function, $p\in \R$, 
$\omega\in S^1$, 
and $\d\ell$ is the Euclidean line measure on each line in the integral above. If $\kappa=1$, we write $\mathcal{R}=\mathcal{R}_\kappa$. 
Each line is represented twice: as $(\omega,p)$ and as $(-\omega,-p)$ but it is represented only once as a directed one. In general, the weight does not need to be even in the $\omega$ variable, so it is natural to think of the lines  as directed ones. 
Let $\omega^\perp$ be $\omega$ rotated by $\pi/2$. We parameterize $\omega$ by its polar angle $\varphi$ and write
\be{omega}
\omega(\phi) = (\cos\varphi,\sin\varphi). 
\ee
 The Schwartz kernel of $\mathcal{R}_\kappa$ is $\kappa\delta(x\cdot\omega(\phi)-p )$.   
Then it is straightforward to show that $\mathcal{R}_\kappa$ is an FIO of order $-1/2$ with canonical relation
\be{R2}
C=\bigg\{\Big(\varphi, \underbrace{x\cdot\omega(\varphi)}_{p}, \underbrace{-\lambda (x\cdot\omega^\perp(\varphi))}_{\hat\varphi},\underbrace{\lambda}_{\hat p}, x, \underbrace{\lambda\omega(\varphi)}_{\hat x=\xi}\Big),\; \lambda\not=0\bigg\},
\ee
where we used the non-conventional notation of denoting the dual variable of $p$ by $\hat p$, etc. 
Set $\xi=\lambda\omega$. Given $\xi\not=0$, there are two solutions for $\lambda$, $\omega$: either $\lambda_+=|\xi|$, $\omega=\xi/|\xi|$ or $\lambda_-=-|\xi|$, $\omega=-\xi/|\xi|$. Therefore, $C=C_+\cup C_-$, where
\be{R2c}
C_\pm:  (x, \xi)\longmapsto \bigg(\underbrace{\arg (\pm\xi)}_\varphi, \underbrace{\pm x\cdot\xi/|\xi|}_p, \underbrace{- x\cdot\xi^\perp }_{\hat\varphi},\underbrace{\pm |\xi|}_{\hat p}\bigg).
\ee
The Schwartz kernel of $\mathcal{R}_\kappa$ is a delta function on $Z:=\{x\cdot\omega(\varphi)=p\}$ which is invariant under the symmetry $\tau(\varphi,p) = (\varphi +\pi,-p)$ (the latter modulo $2\pi$). Then $C$ is invariant under $\tau$ lifted to the cotangent bundle: 
\be{R2D}
(\varphi, p,\hat \varphi,\hat p)\quad \longmapsto \quad (\varphi+\pi, -p,\hat \varphi,-\hat p),
\ee
and in fact this is an isomorphism between $C_+$ and $C_-$. 

Take $C_+$ first. We see that a singularity $(x,\xi)$ of $f$ can create a singularity of $\mathcal{R}_\kappa f(p,\omega)$ at  $p=x\cdot\xi/|\xi|$ and $\omega = \xi/|\xi|$, i.e., at $\varphi=\arg(\xi)$ in the codirections $|\xi|(- (x\cdot\omega^\perp)\omega^\perp,1)$. Note first that such $(p,\omega)$ determines the oriented  line  through $x$ normal to $\xi$ and the normal is consistent with the orientation. Taking $C_-$ next, we see that $(x,\xi)$ may affect the wave front set of $\mathcal{R}_\kappa$ at $p=-x\cdot\xi/|\xi|$ and $\omega=-\xi/|\xi|$, and at the corresponding codirections. That is the same line as before but with the opposite orientation and the weight $\kappa$ on it might be different.

So we see that $\mathcal{R}_\kappa$ is an FIO associated with $C_-\cup C_+$ and each one of them is a local  diffeomorphism. Indeed, $(x,\xi) = C_\pm^{-1}(\varphi,p,\hat\varphi,\hat p)$ is given by
\be{R2A}
 x= p\omega(\varphi) - (\hat\varphi/\hat p)     \omega^\perp(\varphi), \quad   \xi= \hat p\omega(\varphi). 
\ee
It is well defined for $\hat p\not=0$ but if we want $x$ in the image to be in $|x|<R$, we need to require $p^2+(\hat\varphi/ \hat p)^2<R^2$, see also \r{4.1} and  Figures  \ref{fig_R} and \ref{fig_R2}. 

\subsection{(i) Sampling $\mathcal{R}_\kappa f$.} 
\subsubsection{Sampling $\mathcal{R}_\kappa f$ on periodic grids.}
Assume that $f=f_h$ satisfies
\be{R_as}
\WFH(f)\subset \{(x,\xi);\; |x|\le R, \, |\xi|\le B\}
\ee
with some $R>0$, $B>0$, i.e., up to $O(h^\infty)$,   $f$ is essentially supported in $B(0,R)$ and $\mathcal{F}_hf$  is essentially supported in $|\xi|\le B$. The number $N_F$ of points to sample $f$ stably then is given
\be{Nf}
N_f\sim (2\pi h)^{-2}\pi R^2\times \pi B^2 = h^{-2}R^2B^2/4,
\ee
where $\sim$ means equality up to an $(1+o(h))$ relative error, see Theorem~\ref{thm_Landau} and Remark~\ref{rem_L1}.

A sharp upper interval containing   $\hat p$ is $[-|\xi|,|\xi|]$ and a sharp   interval for $\hat\varphi$ is  $[-R|\xi|, R|\xi|]$,
\begin{wrapfigure}[14]{r} {0.46\textwidth}
\begin{center}
	\includegraphics[page=4 ]{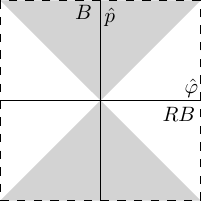}
\end{center}
\caption{\small The frequency set of $\mathcal{R}_\kappa  f$.}
\label{fig_R}
\end{wrapfigure} 
 therefore, the smallest rectangle containing $\mathcal{F}_h \mathcal{R}_\kappa f$ for all such possible $f$'s is 
\[ \hspace{-200pt}
(\hat\varphi, \hat p)\in   RB[-1,1]\times B [-1,1].
\]
Note that this rectangle does not describe  $\Sigma_h(\mathcal{R}_\kappa f)$   when   $\Sigma_h(f)$ is included in $|\xi|\le B$. The latter  is the cone $|\hat\varphi|\le R|\hat p|$ in that rectangle, i.e.,
\be{4.1}\hspace{-200pt}
\left\{(\hat\varphi, \hat p);\; |\hat\varphi|\le R|\hat p|,\, |\hat p|\le B \right\},
\ee
see Figure~\ref{fig_R}.  Suppose we sample on a rectangular grid with 
 sampling rates $s_\varphi h$ in the $\varphi$ variable in $[0,2\pi]$; and with a sampling rate $s_ph$ in the $p$ variable in $[-R,R]$. Then the Nyquist condition is equivalent to
\be{R_N}
s_\varphi \le\frac{\pi}{RB}, \quad s_p\le \frac{\pi}B. 
\ee 
 This means taking $ 2RB/h\times 2 RB/(\pi h)= h^{-2} 4R^2B^2/\pi$ samples to recover $\mathcal{R}_\kappa$ approximately, see also  \cite{Natterer-book}.   Note that this is $8/\pi$ times the estimate in \r{Nf} for $2N_f$, which by Corollary~\ref{cor_num} is the theoretical asymptotic minimum since the canonical relation $C$ is 1-to-2. 
For the recovery of $\WFH(f)\setminus 0$ we need half of those samples, i.e., $\sim h^{-2} 2R^2B^2/\pi$. This is again $8/\pi$ times the asymptotic minimum in Corollary~\ref{cor_num}. 

On the other hand, the conic shape of the range of $\WFH(f)$ in \r{4.1} suggests a more efficient sampling. One can tile the plane with that set by using the 
elementary translations by $(RB,B )$ and by $(0,2B)$. 
If $2\pi (W^*)^{-1}$ has those columns, then 
\be{4.2}
W = \frac{\pi}{RB} \begin{pmatrix} 2 &-1 \\0&R \end{pmatrix}.
\ee
Then one should sample on a grid $s hW\mathbf{Z}^n$ with $s<1$.  Since $\det W=2\pi^2/RB^2$, and the area of the region we sample is $2\pi\times 2R=4\pi R$, we see that we need $\sim h^{-2}2R^2B^2/\pi$ points which is $4/\pi$ larger than $2N_f$; and for proper recovery of $\WFH(f)$ we need a half of that, i.e., $\sim 4/\pi$ times $N_f$. The coefficient $4/\pi$ is close to $1$ but not equal to $1$ as it is clear from the next section and  Figure~\ref{fig_R2}.

\subsubsection{Microlocalization and non-uniform sampling.} \label{sec_NU1} 
The sampling requirements above were based on the following. To determine the sampling rate for $\mathcal{R}_\kappa f$, we find $\Sigma_h(R_\chi f)$ as the projection of $\WFH(\mathcal{R}_\kappa f)\subset C\circ\WFH(f)$ to its phase variable. Since we are interested in sampling $f$ with $\WFH(f)\subset \{|x|\le R, \; |\xi|\le B\}$ and to find a sharp sampling rate, we projected $C\circ \{|x|\le R, \; |\xi|\le B\}$  onto its phase variable to get  the smallest closed set  \r{4.1} containing $\Sigma_h(\mathcal{R}_\kappa f)$ for every such $f$. This answers the question if we are interested in sampling on a periodic grid for all such possible $f$'s. The analysis allows us to localize or microlocalize some of those arguments. 

\textbf{The dependence of the sampling requirements on $\supp f$.} 
The sampling frequency in the angular variable on a rectangular grid should be smaller than $\pi h/(RB)$, and the dependence  on $R$ may look strange since the Radon transform has a certain translation invariance. The reason for it is that we assume that we know that  $f$ is supported in a disk and we reconstruct it there only.
 Numerical experiments reveal that when the sampling rate in the angular variable decreases, artifacts do appear and they move closer and closer to the original when the rate decreases. 

\textbf{Non-uniform sampling.} 
We are interested first in the optimal sampling rate of $\mathcal{R}_\kappa f$ locally, near some $(\varphi,p)$. The latter is determined by the frequency set 
$C\circ \{|x|\le R, \; |\xi|\le B\}$ projected to its phase variables $(\hat\varphi,\hat p)$ with $(\varphi, p)$ fixed. It is straightforward to see that on the range of $C$, 
\be{4.3}
|\hat\varphi| =\pm |\hat p| \sqrt{|x|^2-p^2}, \quad |\hat p|\le B, \quad \text{where $|p|\le |x|$}.
\ee
Since $x$ ranges in $|p|\le |x|\le R$, we get 
\be{4.4}
|\hat\varphi| \le |\hat p|\sqrt{R^2-p^2},\quad  |\hat p|\le B.
\ee
 We plot those double triangles in Figure~\ref{fig_R} at a few points in the rectangle $[0,2\pi]\times [-1,1]$ (where $\varphi=0$ and $\varphi=2\pi$ should be identified) in the $(\varphi,p)$ plane. The phantom $f$ consists of six small Gaussians in the unit disk. 
\begin{figure}[h!] 
  \centering
	\includegraphics[trim = 0 20 0 0, clip,scale=0.3]{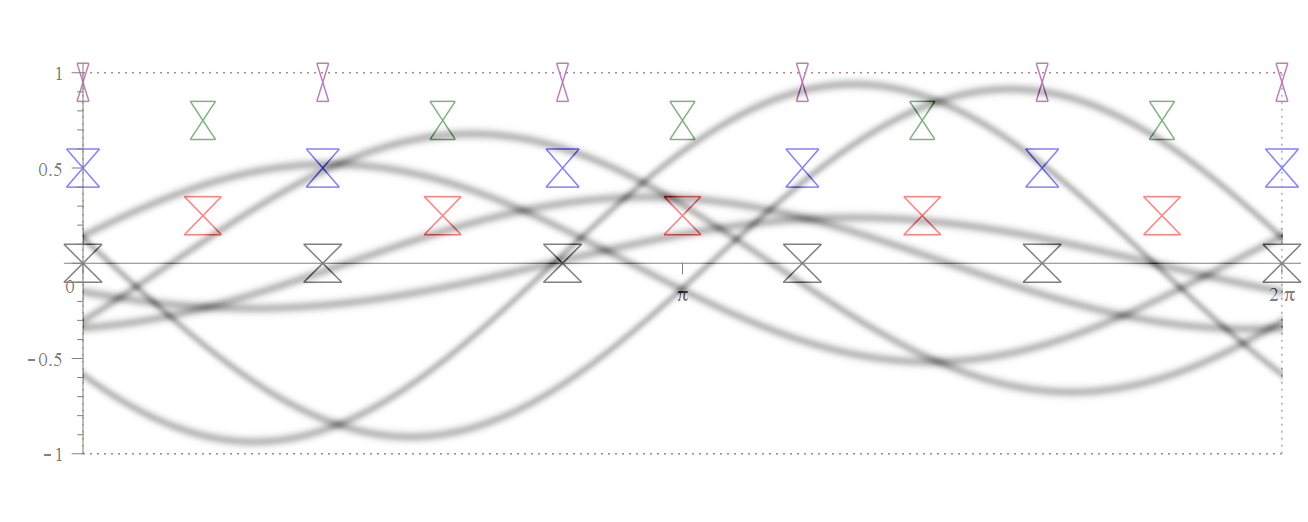}
\caption{\small $ \mathcal{R}f(\varphi,p)$ for $f$ consisting of six randomly placed small Gaussians in the unit ball. The double triangles, shown in the upper half only, represent $\WFH( \mathcal{R}f)$ localized at their vertices $(\varphi,p)$.}
\label{fig_R2}
\end{figure}
We also superimpose a density plot of $\mathcal{R}f$ for a certain $f$ consisting of six randomly places small Gaussians in the unit disk (with $R=1$). 
One can see from this figure that the set of the conormals to the curves in the plot of $\mathcal{R}f$ fall inside those triangles but the \sc\ wave front set also captures the range of the magnitudes of the frequencies. When the stripes are horizontal, the magnitude drops a bit and the stripes a bit more blurred. 
Since $R=1$ in this example, the Nyquist sampling limits of the sampling rates $s_\varphi h$ and $s_ph$ given in \r{R_N} are equal, i.e., the optimal grid would be a square one. 

This analysis suggests the following non-uniform sampling strategy. For a fixed $k$, we divide the interval $[-1,1]\ni p$ into $2k$ sub-intervals $I_j^\pm =\pm [(j-1)/k,j/k]$, $j=1,2,\dots,k$, each one of length $1/k$. For each $k$, we take $|\hat\varphi|\le |\hat p|\sqrt{R^2-j^2/k^2}$, $|\hat p|\le B$ as a sharp cone where $(\hat\varphi, \hat p)$ lies, see \r{4.3}. Then in $[0,2\pi]\times I_j$, we sample on the grid $\delta h W_j\mathbf{Z}^n$, where $\delta<1$ is fixed and $W_j$ is as in \r{4.2} with $R=1/j$. Then we can get closer and closer to the sharp number of the sampling point for $\mathcal{R}_\chi$ stated in Corollary~\ref{cor_num} and Theorem~\ref{thm_Landau}, which should be $~\sim 2N_f$ for a stable recovery of $R_\kappa f$ and $N_f$ for a stable recovery of $f$ itself, see \r{Nf}.

\subsection{(ii) Resolution limit on $f$ posed by the sampling rate of $\mathcal{R}_\kappa f$} 
 Let $s_\varphi$ and $s_p$ be the relative sampling rates for $\varphi$ and $p$, respectively. Lack of aliasing is equivalent to $\hat\varphi<\pi/s_\varphi$, $\hat p<\pi/s_p$, see \r{C1}, \r{C2}. By \r{R2c}, this is equivalent to 
\be{Res1}
|x\cdot\xi^\perp|\le \pi/s_\phi,\quad |\xi|\le \pi/s_p.
\ee
If the sampling rates satisfy the sharp Nyquist condition \r{R_N}, the latter condition above implies the former. Actually, the first condition in \r{Res1} is most critical for $(x,\xi)$ with $x$ close to the boundary $|x|=R$ and $\xi\parallel x$, which  are represented by radial lines close to $|x|=R$. In Figure~\ref{fig2c}, which is undersampled in $\varphi$, we see evidence of that; another evidence is  Figure~\ref{fig3c}, where $\mathcal{R}f$ is blurred in $\varphi$. 

When the sampling rates do not necessarily satisfy the Nyquist condition \r{R_N} in $|x|<R$, we illustrate the significance of \r{Res1} in Figure~\ref{fig_R_resolution}. The relative sampling rate $s_p$ imposes a universal limit on the resolution, independent on $x$ and the direction of $\xi$. On the other hand, the second inequality imposes a locally non-uniform and a non-isotropic resolution limit. Assuming  $s_p\ll1$ (which is true in practical applications), 
in optics terms, the resolution of saggital (radial) lines deteriorates gradually away from the center; there, $|x\cdot \xi/|\xi||\ll1$, so $|x\cdot \xi^\perp/|\xi^\perp||$ is close to its maximum for that $x$, which restricts $|\xi|$ by \r{Res1}. Resolution of meridional (circular) lines is the greatest; there $|x\cdot \xi^\perp/|\xi^\perp||\ll1$, so for a given $x$, $|\xi|$ could be large by \r{Res1}. There are also aliasing artifacts explained below. 

\begin{figure}[h!] \hfill
\begin{subfigure}[t]{.22\linewidth}
\centering
\includegraphics[height=.15\textheight]{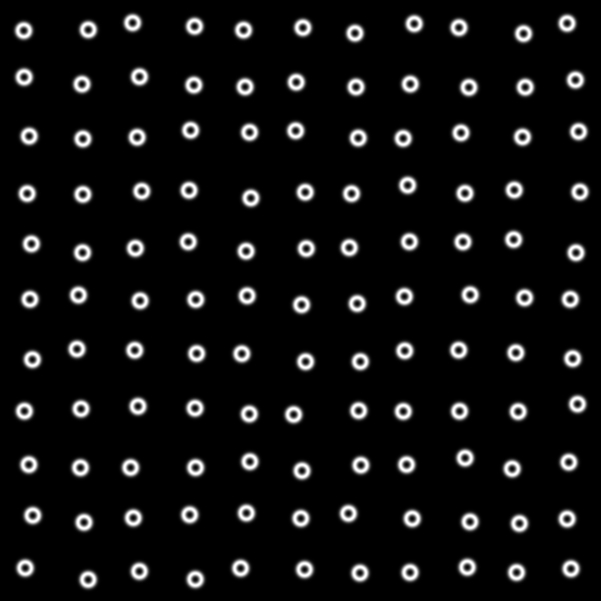}
\caption{$f$ on $[-1,1]^2$}\label{fig2a}
\end{subfigure}\hfill
\begin{subfigure}[t]{.22\linewidth}
\centering
\includegraphics[height=.15\textheight]{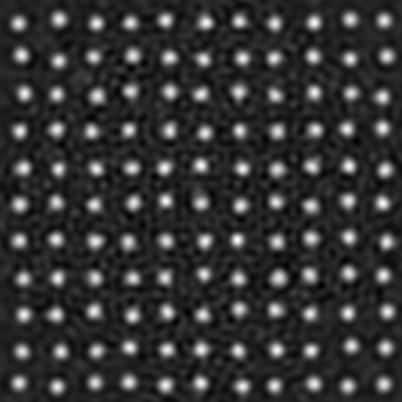}
\caption{Reconstructed $f$ undersampled in $p$.}\label{fig2c}
\end{subfigure}  \hfill 
\begin{subfigure}[t]{.22\linewidth}
\centering
\includegraphics[height=.15\textheight]{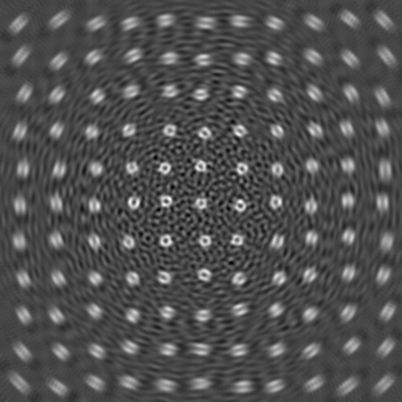}
\caption{Reconstructed $f$,  angular step $s_\varphi=4^\circ$}\label{fig2b}
\end{subfigure} \hfill \,\vspace{-10pt}
\caption{(a) $f$ on $[-1,1]^2$, (b) $f$ reconstructed with $s_p\ll1$ and an angular step  $s_\varphi=4$ degrees; (c) $f$ reconstructed with $s_\varphi\ll1$ and undersampled in the $p$ variable. 
}  \label{fig_R_resolution}
\end{figure}

\subsection{(iii) Aliasing} We study now what happens if $\mathcal{R}_\kappa f$ is undersampled. It might be undersampled in the $\varphi$ or the $p$ variable or in both.

\subsubsection{Angularly undersampled $\mathcal{R}_\kappa f$} Assume \r{R_as} as above and assume that $s_\varphi> \pi h/(RB) $, i.e., the first Nyquist condition in \r{R_N} is violated. Then the aliasing of $\mathcal{R}_\kappa f$ can be described as a sum of   h-FIOs, see \r{A2}, \r{sh}, with canonical relations
\be{sg_phi}
S_k : \hat\varphi\longmapsto \hat\varphi+ {2\pi k}/{s_\varphi} , \quad k=0,\pm1, \dots . 
\ee
In typical cases with not very severe undersampling, $k$ is restricted to $k=\pm 1$ plus $k=0$ which is the original image but blurred by $\hat\chi$ in \r{A2}. 
Then  a direct computation shows that the aliasing artifacts are described by an h-FIO with canonical relations 
\be{RA1}
(x,\xi)\quad\longmapsto \quad C_\pm^{-1}\circ S_k\circ C_\pm (x,\xi) = \Big(x \mp \frac{2\pi k}{s_\varphi} \frac{\xi^\perp}{|\xi|^2} ,  \xi\Big),
\ee
when $\hat\varphi+2k\pi/s_\varphi\in [-\pi/s_\varphi, \pi/s_\varphi]$, i.e., when \be{RA2}
-x\cdot\xi^\perp+2k\pi/s_\varphi\in [-\pi/s_\varphi, \pi/s_\varphi].
\ee
Those are shifts of $(x,\xi)$ in the $x$ variable, in the direction of $\xi^\perp$, at distance $2\pi k/(s_\varphi |\xi|)$. By \r{RA2}, $k$ depends on $(x,\xi)$ and in particular for $|x\cdot\xi^\perp|\ll1$, we have $k=0$ only of $\Sigma_h(f)$ is finite and then there is no aliasing. In general, the reconstructed $f$ will have the singularities of $f$ shifted by \r{RA1} for various $k=0,\pm1,\dots$, as long as they satisfy \r{RA2}. The value $k=0$ corresponds to $\WFH(f)$ (not shifted). Note that only finitely many of them would stay in the ball $|x|<R$.  It is even possible all of them to be outside that ball and $\mathcal{R}_\kappa f$ to be undersampled and therefore aliased. Then the reconstructed image in $|x|<R$ will not have a singularity corresponding to that one.

\begin{figure}[h!]\,\hfill
\begin{subfigure}[t]{.22\linewidth}
\centering
\includegraphics[height=.15\textheight]{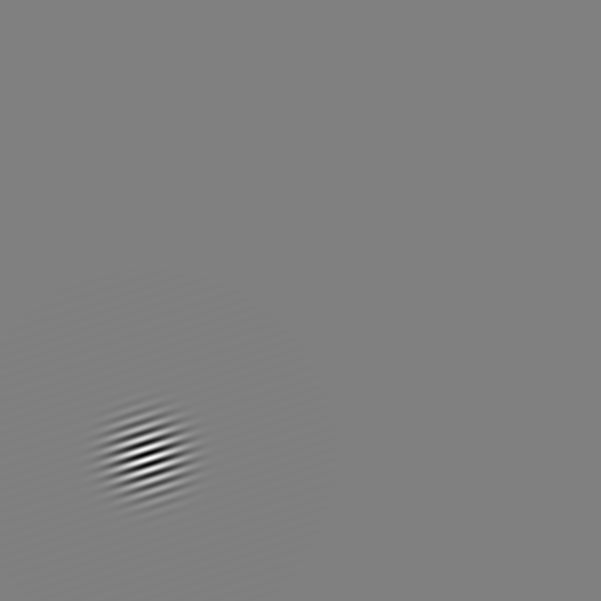}
\caption{$f$ }\label{fig1a}
\end{subfigure}\hfill
\begin{subfigure}[t]{.39\linewidth}
\centering
\includegraphics[height=.15\textheight]{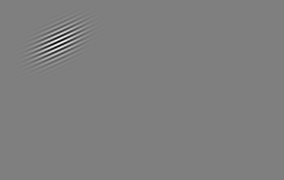}
\caption{$\mathcal{R}f$ on $[\pi/2,\pi]\times[0,1/2]$ oversampled}\label{fig1b}
\end{subfigure}\hfill
\begin{subfigure}[t]{.22\linewidth}
\centering
\includegraphics[height=.15\textheight]{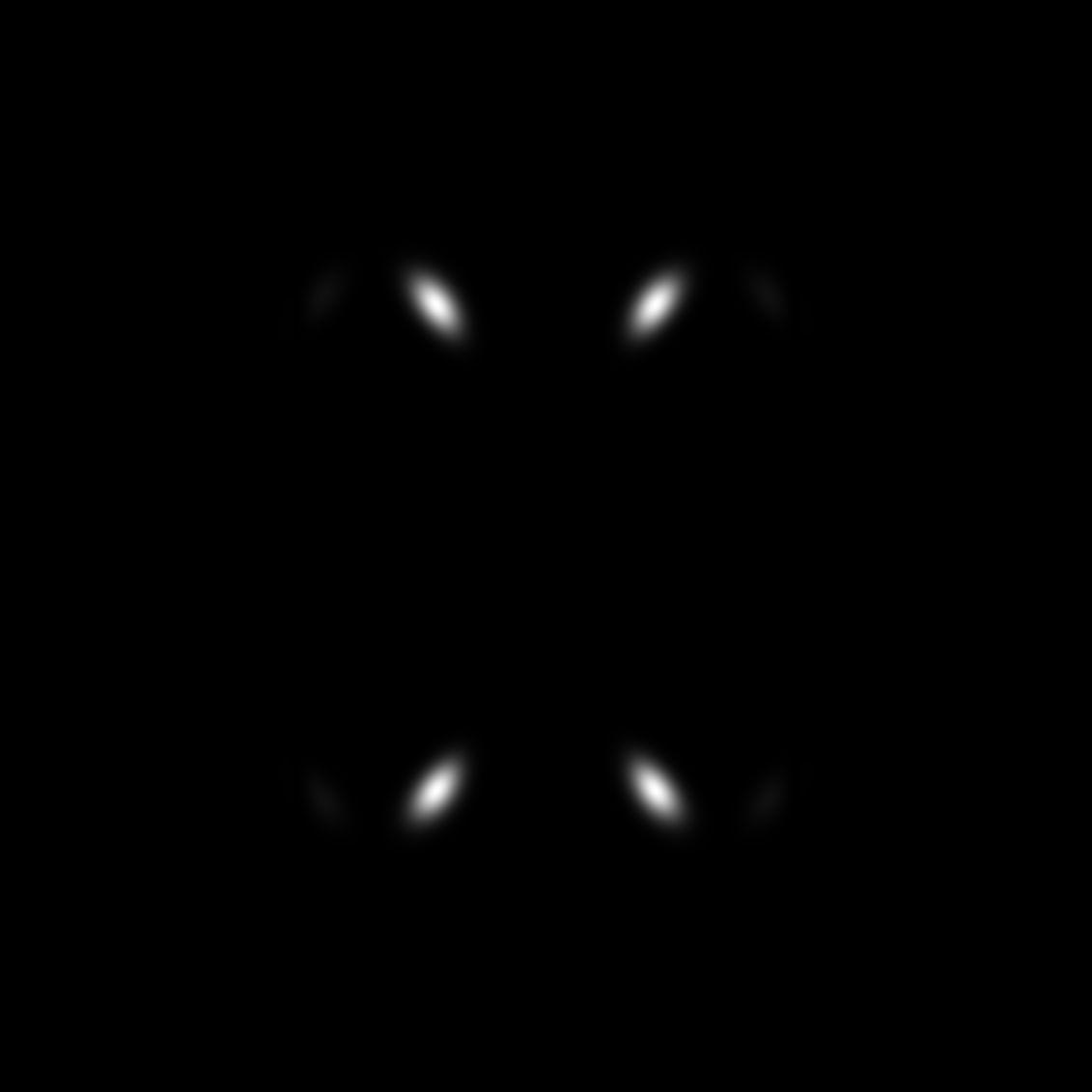}
\caption{$\mathcal{FR}f$ oversampled}\label{fig1c}
\end{subfigure}\hfill \,\\ \vspace{10pt}
\,\hfill
\begin{subfigure}[t]{.22\linewidth}
\centering
\includegraphics[height=.15\textheight]{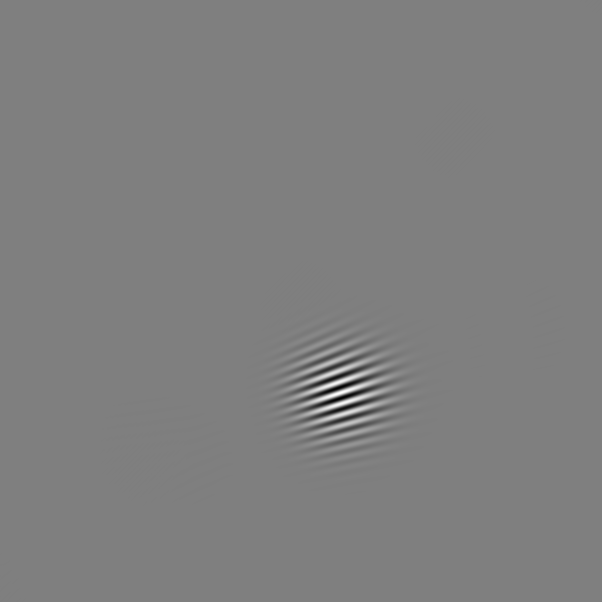}
\caption{$f$ recovered}\label{fig_R_aliasing2a}
\end{subfigure}\hfill
\begin{subfigure}[t]{.39\linewidth}
\centering
\includegraphics[height=.15\textheight]{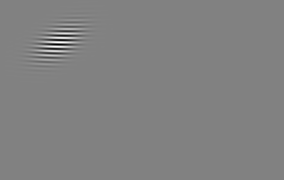}
\caption{$\mathcal{R}f$ on $[\pi/2,\pi]\times[0,1/2]$ undersampled}\label{fig_R_aliasing2b}
\end{subfigure}\hfill
\begin{subfigure}[t]{.22\linewidth}
\centering
\includegraphics[height=.15\textheight]{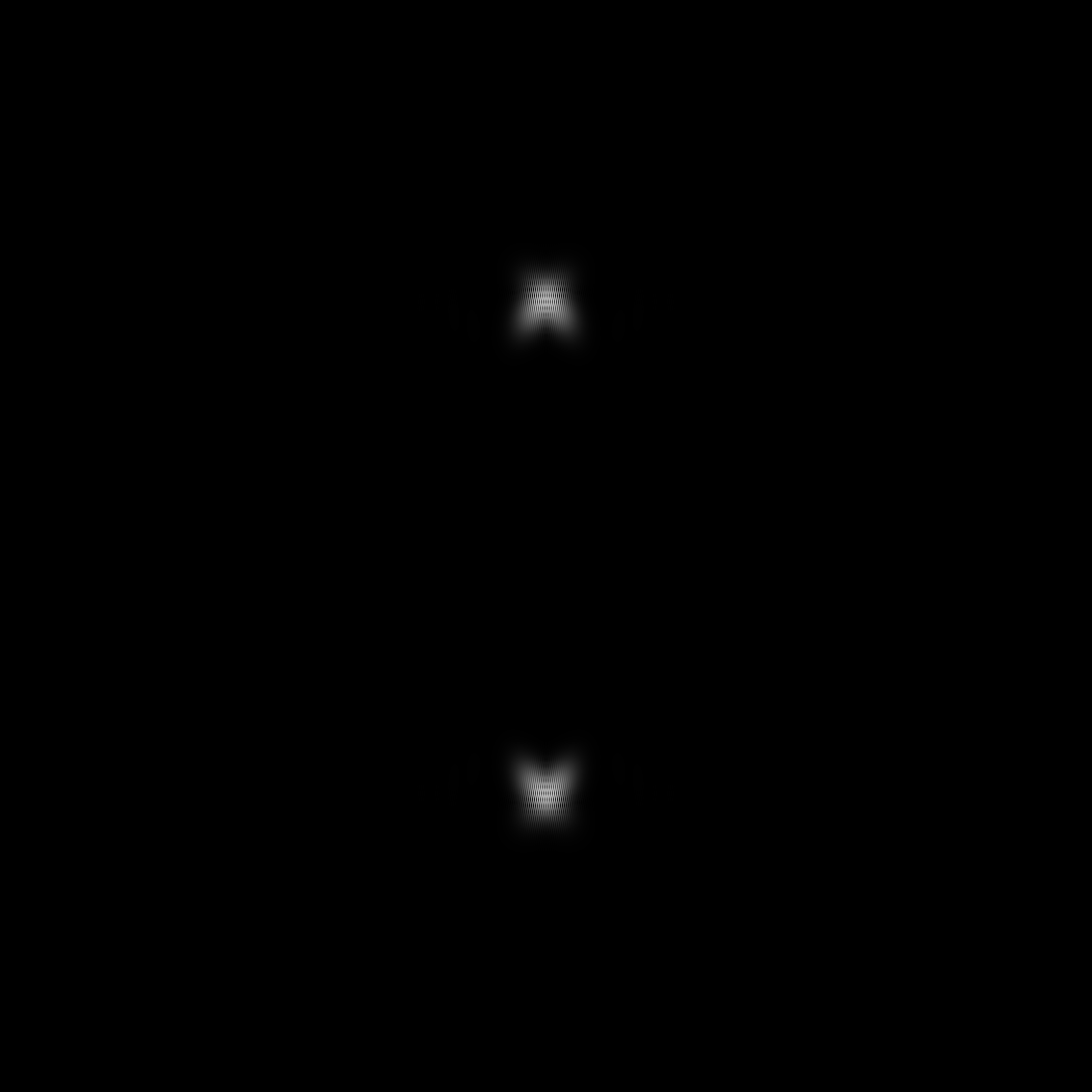}
\caption{$\mathcal{FR}f$ undersampled}\label{fig_R_aliasing2c}
\end{subfigure}\hfill \,
\caption{Top: $f$, $\mathcal{R}f$ and the Fourier transform $\mathcal{FR}f$ of $\mathcal{R}f$ when $\mathcal{R}f$ is angularly undersampled. Bottom: 
 $f$ reconstructed with a back-projection, $\mathcal{R}f$ and $\mathcal{FR}f$  when $\mathcal{R}f$ is angularly undersampled with a $3$ degrees step: the pattern has shifted. 
}\label{fig_R_aliasing2}
\end{figure}

We illustrate this with a numerical example in Figure~\ref{fig_R_aliasing2a}. We choose $f$ to be a coherent state as in Figure~\ref{fig_aliasing}. In Figure~\ref{fig_R_aliasing2}, we plot $f$, a crop of its Radon transform $\mathcal{R}f$, oversampled, on $[\pi/2,\pi]\times [0,1/2]$ (the only other significant part is symmetric to it and we do not show it), and the Fourier transform of $\mathcal{FR}f$. Since $\mathcal{R}f$ is even and real valued, $\mathcal{FR}f$ has two symmetries. A reconstruction of $f$ with oversampled data, not shown, looks almost identical to $f$. Next we undersample $\mathcal{R}f$ using a $s_\varphi=3$ degree step in $\varphi$. 
In  Figure~\ref{fig_R_aliasing2}d, we show the reconstructed $f$ which looks like $f$ shifted along the direction of the pattern. The undersampled $\mathcal{R}f$ used to get reconstruction is shown in (b). 
Compared to (e), the pattern changed its orientation (and the magnitude of its frequency), similarly to the classical aliasing effect illustrated in Figure~\ref{fig_aliasing}. The effect on the reconstructed $f$, see (d), however is very different and in an agreement with \r{RA1}. In (f), we plot $\mathcal{FR}f$ where  $\mathcal{R}f$ now is the aliased version of the Radon transform of $f$. We see that the bright spots where $\mathcal{R}f$ is essentially supported have shifted compared to (c): the ones to the left have shifted to the right and vice versa, as explained earlier.

In this case, only the values $k=\pm1$ in \r{RA3} contribute to singularities because the singularity of $f$ does not satisfy \r{RA4} with $k=0$, i.e., the original singularity is not within the resolution range. 

A similar example, not shown, with the pattern moved close to the center is reconstructed well (see also Figure~\ref{fig_R_resolution}c) is reconstructed well without an artifact even though the artifact computed by \r{RA1} would still fit in the square shown. The reason for it is condition \r{RA2} which for small $|x|$ and the  other parameters unchanged is valid for $k=0$ only. 

\subsubsection{$\mathcal{R}f$ undersampled  in the $p$ variable} Assume that $s_p$ is not small enough to satisfy the sampling conditions but $s_\varphi$ is. The aliasing of $\mathcal{R}_\kappa f$ then can be computed, using \r{R2A} and \r{sh}, to be 
\be{RA3}
(x,\xi)\quad\longmapsto \quad   \Big(x \pm  x\cdot\xi^\perp \left(   \frac1{|\xi|+ 2\pi k/s_p}  -\frac1{|\xi|} \right) 
\frac{\xi^\perp}{|\xi|} ,  \xi+    \frac{2\pi k}{s_p}\frac{\xi}{|\xi|} \Big),
\ee
when $\hat p+(2\pi k/s_p )\in [-\pi/s_p, \pi/s_p]$, i.e., when 
\be{RA4}
|\xi|+2k\pi/s_p\in [-\pi/s_p, \pi/s_p].
\ee
Those are still shifts along $\xi^\perp$ but they are not equally spaced (with $k$). Also, the magnitude of the frequency changes but the direction does not. 
In case of mild aliasing, we have $k=\pm1$ (when we are recovering $f$ in $|x|<R$) and they generate  shifts of different sizes. In general, there are infinitely many artifacts outside the ball $|x|<R$ regardless of the sampling rate and the band limit of $f$ (our criterion whether $Rf$ is aliased or not depends on $R$). 

In Figure~\ref{fig_R_aliasing5}, we present an example where one of the patterns disappears from the computational domain $[-1,1]^2$ due to undersampling in the $p$ variable. The other one remains. 

\begin{figure}[h!]\,\hfill
\begin{subfigure}[t]{.2\linewidth}
\centering
\includegraphics[height=.13\textheight]{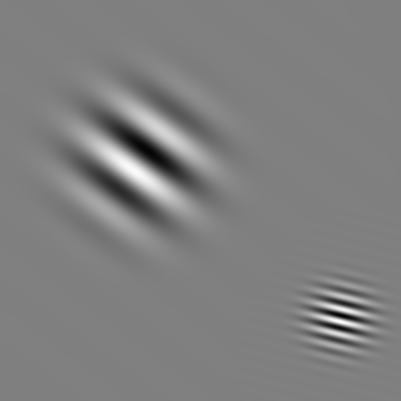}
\caption{$f$ }\label{fig5a}
\end{subfigure}\hfill
\begin{subfigure}[t]{.2\linewidth}
\centering
\includegraphics[height=.13\textheight]{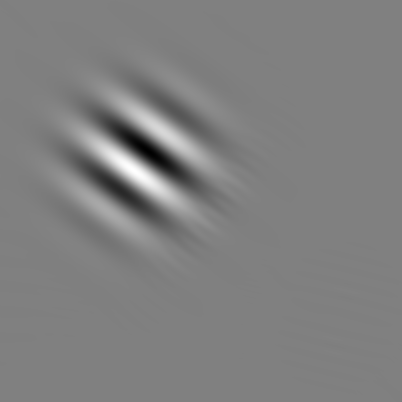}
\caption{$f$ reconstructed}\label{fig5b}
\end{subfigure}\hfill
\begin{subfigure}[t]{.28\linewidth}
\centering
\includegraphics[height=.13\textheight]{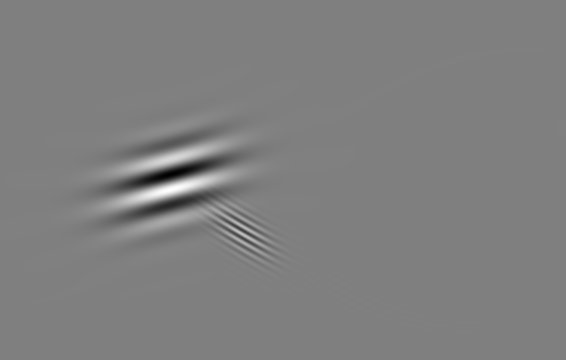}
\caption{$\mathcal{R}f$ on $[0,\pi]\times[-1,1]$ oversampled}\label{fig5c}
\end{subfigure}\hfill
\begin{subfigure}[t]{.28\linewidth}
\centering
\includegraphics[height=.13\textheight]{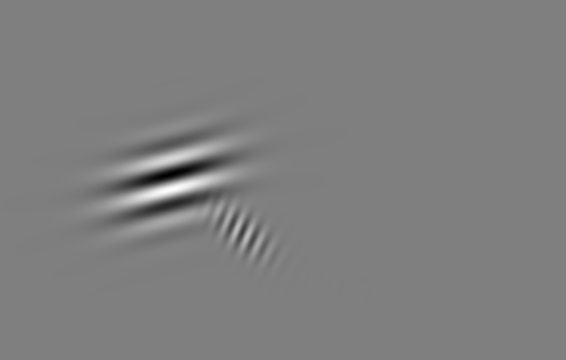}
\caption{$\mathcal{R}f$ on $[0,\pi]\times[-1,1]$ undersampled}\label{fig5d}
\end{subfigure}\hfill
\vspace{-10pt}
\,\hfill
\caption{(a) $f$ and (b) reconstructed $f$ with $\mathcal{R}f$ undersampled in $p$. (c)   $\mathcal{R}f$ and (d) $\mathcal{R}f$ aliased. 
}\label{fig_R_aliasing5}
\end{figure}

\subsection{(iv) Locally averaged measurements} \label{sec_R_AI}
Assume now that we measure $Q_h \mathcal{R}_\kappa f$ with $Q_h$ an \HPDO\ of order $(0,0)$ (or simply a convolution) limiting the frequency set $\Sigma_h(\mathcal{R}_\kappa f)$ of the data. If $q_0(\varphi,p,\hat p,\hat \varphi)$ is the principal symbol of $Q_h$, then by Proposition~\ref{prop_av},  a backprojection reconstructs  $P_hf$ where  $P_h$ has a principal symbol 
\be{RAI1}
p_0(x,\xi) = \frac12 q_0\circ C_+ + \frac12 q_0\circ C_-.
\ee
If, in particular, $Q_h$ is a convolution with  a 
kernel of the type $q_0=\psi(a\hat\varphi^2+b\hat p^2)$ with  $\varphi\in C_0^\infty$ and decreasing, then 
\be{Av1}
p_0(x,\xi) = \psi\big(a|\xi|^2 + b|x\cdot\xi^\perp|^2\big).
\ee
This symbol takes its smallest  values for $x$ near the boundary and $\xi\perp x$, and those are the covectors with the lowest resolution as well. The effect of $Q_h$ is then non-uniform, it blurs $f$  the most at those covectors. If we want a uniform blur, then we choose $p_0=\psi(a|\xi|^2)$ and compute $q_0=\psi(a\hat p^2)$. This is not surprising in view of the classical intertwining property $\d_p^2 \mathcal{R}_\kappa= \mathcal{R}_\kappa\Delta$ when $\kappa=1$ (true modulo lower order terms for general $\kappa$).  In other words, only convolving w.r.t.\ the $p$ variable is needed. This means integrating over ``blurred lines''. 
If $\psi$ limits $\WFH(Q_h \mathcal{R}_\kappa f)$ to, say, $|\hat p|\le B'$, then this limits $\hat\varphi$ as well by the first inequality on \r{4.1}, to $|\varphi|\le |\hat p|$. Therefore, $(\hat\varphi, \hat p)$ are restricted to a smaller cone of the type \r{4.1} which imposes sampling requirements as above. Then we can recover stably $P_hf$. 

In Figure~\ref{fig_R_average}, we show a reconstructed image with data averaged in the $p$ variable (then $b=0$ in \r{Av1}) and the angular variables (then $a=0$ in \r{Av1}) . Note that in Figure~\ref{fig3c} the image is blurred  angularly but in contrast to Figure~\ref{fig2c}, there are no aliasing artifacts. 

\begin{figure}[h!]{} 
\hfill
\begin{subfigure}[t]{.22\linewidth}
\centering
\includegraphics[height=.15\textheight]{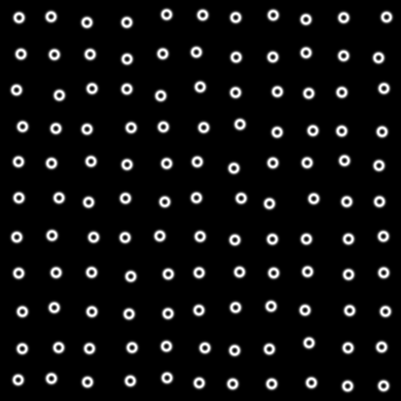}
\caption{$f$ on $[-1,1]^2$}\label{fig3a}
\end{subfigure}\hfill
\begin{subfigure}[t]{.22\linewidth}
\centering
\includegraphics[height=.15\textheight]{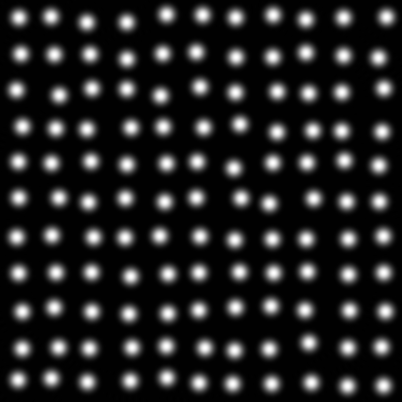}
\caption{reconstruction with $\mathcal{R}f$ averaged in $p$}\label{fig3b}
\end{subfigure}\hfill
\begin{subfigure}[t]{.22\linewidth}
\centering
\includegraphics[height=.15\textheight]{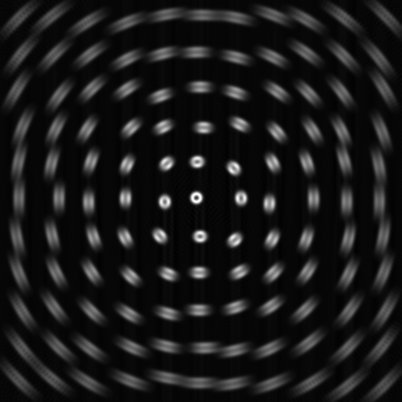}
\caption{reconstruction with $\mathcal{R}f$ averaged in $\varphi$}\label{fig3c}
\end{subfigure}
\hfill \, 
\caption{$f$ and a reconstructed $f$ on $[-1,1]^2$ with data averaged in the $p$ and the $\varphi$ variable.
}  \label{fig_R_average}
\end{figure}

\section{The X-ray/Radon transform $R$ in the plane in fan-beam coordinates} \label{sec_FB}
\subsection{$\mathcal{R}_\kappa$ as an FIO} 
We parametrize $\mathcal{R}_\kappa$ now by the so-called fan-beam coordinates. Each line is represented by an initial point $R\omega(\alpha)$ on the boundary of $B(0,R)$, where $f$ is supported, and by an initial direction making angle $\beta$ with the radial line through the same point,  see Figure~\ref{fig_R3}. It is straightforward to see that this direction is given by  $\omega(\alpha+\beta)$. Then the lines through $B(0,R)$ are given by 
\be{R2a}
x\cdot\omega(\alpha+\beta-\pi/2) = R\sin\beta, \quad \alpha\in [0,2\pi), \; \beta\in [-\pi/2,\pi/2].
\ee
\begin{figure}[h!] 
\begin{center}
	\includegraphics[page=5]{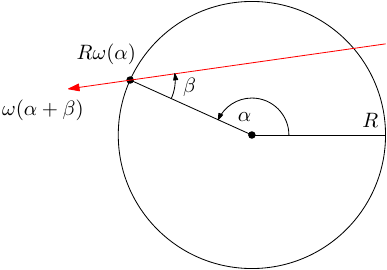}
\end{center}
\caption{\small The fan-beam coordinates.}
\label{fig_R3}
\end{figure}
This allows us to conclude that in this representation $\mathcal{R}_\kappa$ is an FIO again (being FIO is invariant under diffeomorphic changes) and to compute its canonical relation using the rules of transforming covectors. We will do it directly however. 
We regard $\alpha$ as belonging to $\R$ modulo $2\pi\mathbf{Z}$. 
The relationship between this and the parallel geometry parameterization $(\varphi,p)$ is given by
\be{R1a}
\varphi =\alpha+\beta-\pi/2, \quad p=R\sin\beta.  
\ee
Each undirected line is given by a pair $(\varphi,p)$ and $(\varphi+\pi,-p)$; which in the parallel beam coordinates corresponds to $(\alpha,\beta)$ and $(\alpha +2\beta-\pi,-\beta)$. 
The Schwartz kernel of $\mathcal{R}_\kappa$ in this parameterization is a smooth factor times a delta function on the manifold \r{R2a}. 
As above, when $\kappa=1$, we write $\mathcal{R}=\mathcal{R}_\kappa$.  Then 
\be{symm1}
\mathcal{R}(\alpha,\beta) = \mathcal{R}(\alpha +2\beta-\pi,-\beta).
\ee
In general, that change of $(\alpha,\beta)$ is a symmetry of \r{R2a}. 
The canonical relation is given by 
\be{R3}
C=\bigg\{\Big(\alpha,\beta, \underbrace{\lambda (- x\cdot\omega( \alpha+\beta ))  }_{\hat \alpha}, \underbrace{\lambda R\cos\beta-\lambda x\cdot\omega( \alpha+\beta )  }_{\hat\beta}, x, \underbrace{\lambda\omega (\alpha+\beta-\pi/2)  }_{\hat x=\xi}\Big),\; \lambda\not=0\bigg\}.
\ee 
Therefore, with $\omega=\omega(\alpha+\beta)$, we have  $\xi = -\lambda\omega^\perp$, $\hat\alpha = -\lambda    \omega\cdot x= x\cdot\xi^\perp$, $\hat\beta = \lambda R\cos\beta +\hat\alpha$. If $\lambda>0$, then $\lambda=|\xi|$ and then 
$\lambda R\cos\beta=|\xi|\sqrt{R^2-(x\cdot \xi/|\xi|)^2}$. 
Also, by \r{R2a}, $x\cdot\xi= R|\xi|\sin\beta$. 
For the dual variables, we have $\hat\beta= |\xi|\sqrt{R^2-(x\cdot \xi/|\xi|)^2}+\hat\alpha$.

\begin{figure}[h!]\hfill
\begin{subfigure}[t]{.26\linewidth}
\centering
\includegraphics[scale=1.5, page=9]{sampling_pics}
\caption{$f$ on $[-1,1]^2$}\label{fig4a}
\end{subfigure}\hfill
\begin{subfigure}[t]{.49\linewidth}
\centering
\includegraphics[scale=1.5,page=8]{sampling_pics}
\caption{$\mathcal{R}f$}\label{fig4b}
\end{subfigure}  \hfill 
\begin{subfigure}[t]{.15\linewidth}
\centering
\includegraphics[scale=0.5625]{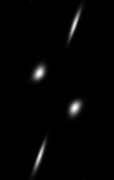}
\caption{$\widehat{\mathcal{R}f}$}\label{fig4c}
\end{subfigure} \hfill  
\caption{The canonical relation of $\mathcal{R}$ in fan-beam coordinates. (a) a coherent state $f$. (b): the image of $(x,\xi)\in \WFH(f)$ under $C_+$ and $C_-$.
}  \label{fig_R4}
\end{figure}

If $\lambda<0$, we get another solution by formally replacing $|\xi|$ by $-|\xi|$. Therefore, the  canonical relations $C_\pm$ are given by
\be{R4}
\beta = \pm \sin^{-1} \frac{x\cdot\xi}{R|\xi|}, \quad \alpha = \arg\xi 
-\beta\pm \frac\pi2 \quad
\hat\alpha = x\cdot\xi^\perp, \quad   \hat\beta= \pm |\xi|\sqrt{R^2-(x\cdot \xi/|\xi|)^2}+\hat\alpha.
\ee
Then $C_\pm$ are isomorphic under the symmetry mentioned above lifted to the tangent bundle
\be{symm2}
(\alpha, \beta, \hat\alpha,\hat\beta) \quad \longmapsto \quad (\alpha+2\beta-\pi , -\beta, \hat\alpha,2\hat\alpha-\hat\beta).
\ee
 We illustrate the canonical relations on Figure~\ref{fig_R4}. On Figure~\ref{fig_f_gaussian}, $\pi_2\circ C_\pm(x,\xi)$ are marked by crosses. 

The inverses $C_\pm^{-1}$ are given by
\be{R4cc}
x = R\sin\beta \,\omega(\alpha+\beta-\pi/2) - \frac{\hat\alpha } { \hat\beta-\hat\alpha}R\cos\beta \,\omega(\alpha+\beta), \quad 
\xi =\frac{ \hat\beta-\hat\alpha}{R\cos\beta} \omega(\alpha+\beta-\pi/2) .
\ee
In particular, we recover the well known fact that $C$ is 1-to-2, as in the previous case.

\subsection{(i) Sampling} 
We assume \r{R_as} again. 

\subsubsection{Sampling on a rectangular lattice.} 
The smallest rectangle including the range of $\hat\alpha$ and $\hat\beta$ if $|x|\le R$, $|\xi|\le B$ is 
\be{R4a}
RB[-1,1    ]     \times   2 RB[-1,1].
\ee
Therefore, 
for the relative sampling rates $s_\alpha$ and $s_\beta$ in the $\alpha$ and in the $\beta$ variables, respectively,   in $[0,2\pi)\times [-\pi/2, \pi/2]$, we have the Nyquist limits
\be{s_ab}
s_\alpha<\frac{\pi}{RB}, \quad s_\beta< \frac{\pi}{2RB},
\ee
compare with \r{R_N}. 
 This means taking more than $2RB/h\times 2RB/h$ samples. This is $\pi$ times more than in the parallel geometry case. 
For a recovery of $\WFH(f)\setminus 0$, we need  a half of that. 

To analyze the actual range, it is enough to analyze the range of $  (\hat\alpha,\hat\beta')=   (\hat\alpha,\hat\beta - \hat\alpha)$, i.e., the l.h.s.\ of \r{C1}. 
Notice first that on $C$, one can parameterize the line corresponding to $(\alpha,\beta)$ as 
\be{range1}
x = R\omega(\alpha) -t\omega(\alpha+\beta), \quad 0\le t\le 2R\cos\beta. 
\ee
Then
\[
\hat\alpha = -|\xi|(R\cos\beta-t), \quad \hat\beta' = \pm |\xi| R\cos\beta, \quad 0\le t\le 2R\cos\beta. 
\]
Therefore, for a fixed $(\alpha,\beta)$, the range of $(\hat\alpha,\hat\beta')$ is independent of $\alpha$ and when $\xi$  varies over $|\xi|\le B$, that range fills the double triangle  $|\hat\alpha|\le |\hat\beta'|\le RB \cos\beta$. Over the whole range of $\beta$, see \r{R2a}, this fills  $|\hat\alpha|\le |\hat\beta'|\le RB $. Then we can get the range of  $  (\hat\alpha,\hat\beta)$, we take the inverse linear transformation. 

\begin{figure}[h!] 
  \centering\hfill
	\includegraphics[page=1]{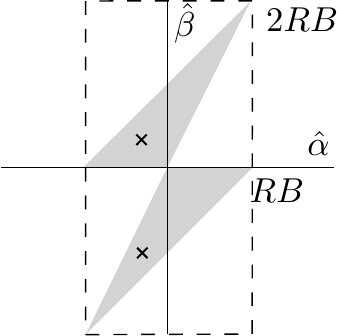}\hfill
  \includegraphics[height=.158\textheight]{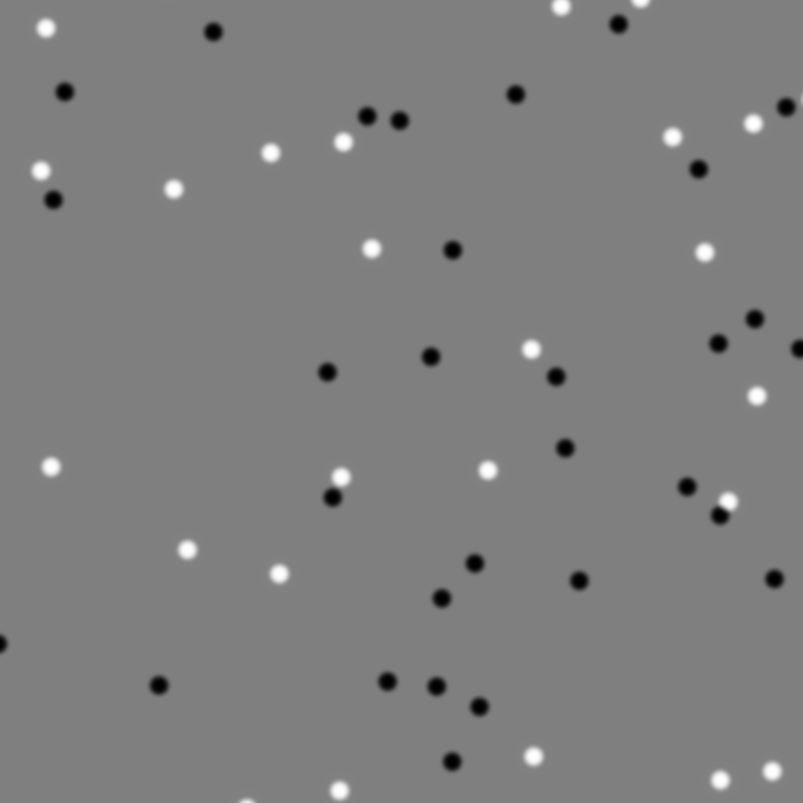}\hfill
    \includegraphics[ height=.158\textheight]{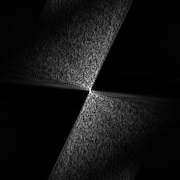}\hfill \, 
\caption{\small Left: the range of $\mathcal{F}_hR$ for $|\xi|\le B$ in fan-beam coordinates.  Center: $f$ as a sum of  randomly placed small Gaussians with mean value zero. 
Right: $|\mathcal{F}_hf|$ in fan-beam coordinates on a log scale.}
\label{fig_f_gaussian}
\end{figure}
In Figure~\ref{fig_f_gaussian}, we show the range for $|\xi|\le B$, and a numerically computed $|\mathcal{F}_hf|$ of $f$ representing a sum of several well concentrated randomly placed Gaussians. 
It has the symmetry \r{symm2}. 
This result can be obtained from the parallel geometry analysis, of course, see \r{4.1} and Figure~\ref{fig_R}, by the change of variables on the cotangent bundle induced by \r{R2a}.

As in the previous case, we can tile the plane with the regions in Figure~\ref{fig_f_gaussian} on the left by taking translations by $(RB,0)$ and $(0,2RB)$. If $2\pi (W^*)^{-1}$ has those columns, then 
\be{R4.2}
W = \frac{\pi}{RB} \begin{pmatrix} 2 &0 \\0&1 \end{pmatrix}.
\ee
Then by Theorem~\ref{thm_sc}, the most efficient sampling would be  on a grid $ shW\mathbf{Z}^n$ with $s<1$, see also \cite{ Natterer-book, Natterer-sampling1993}. This is $\sim 4N_f$, see \r{Nf} and is twice as sparse in each dimension compared to the previous criterion. 
 For a recovery of $\WFH(f)\setminus 0$, we need  a half of that, i.e., $\sim 2N_f$, which is twice as much as the sharp bound in Corollary~\ref{cor_num}.  Note that this however requires a reconstruction formula  of the type \r{2.1} with $\chi$ there having a Fourier transform supported in the gray region in Figure~\ref{fig_f_gaussian} on the left, and equal to one on $\WFH(f)$ instead of the formula based on the sinc functions. The reason that the number of points is not $\sim N_f$ is clear from the analysis below and from Figure~\ref{Radon_FB} as well.

\begin{figure}[h!] 
  \centering
	\includegraphics[trim = 0 0 0 0, clip,scale=0.3]{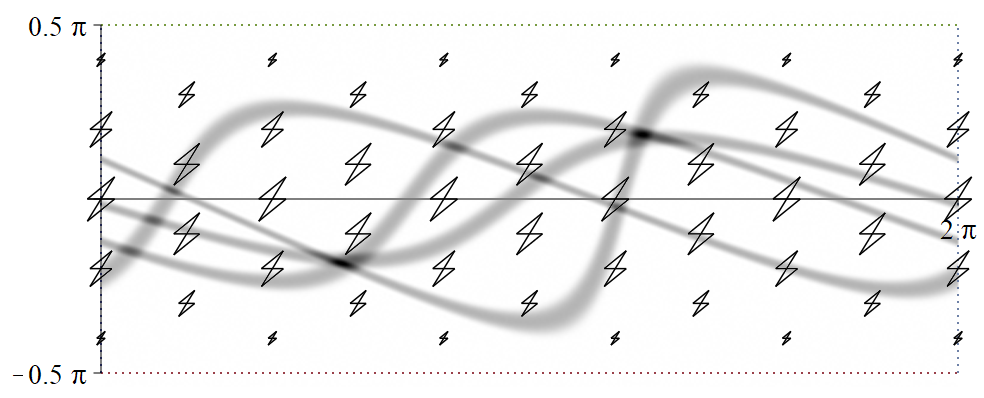}
\caption{\small $ \mathcal{R}f(\alpha,\beta)$ for $f$ consisting of four randomly placed small Gaussians in the unit ball. The double triangles, shown in the upper half only, represent $\WFH( \mathcal{R}f)$ localized at their vertices $(\alpha,\beta)$.}
\label{Radon_FB}
\end{figure}

In Figure~\ref{Radon_FB}, we plot $\mathcal{R}f(\alpha,\beta)$ on $[0,2\pi]\times [-\pi/2,\pi/2]$ for $f$ consisting of four small Gaussians. We also plot the range of $\WFH(f)$ for all possible $f$ satisfying \r{R_as} at each $(\alpha,\beta)$, i.e, we plot \r{range1}. The double triangles represent the set of all possible conormals of singularities in $\mathcal{R}f(\alpha,\beta)$ with their lengths. As we can see (and prove), the highest oscillations can occur on the $\beta=0$ line and they are along the direction $(1,2)$. If we do non-uniform sampling, this is where we need the highest rate. This is also confirmed by the shape and the thickness of the stripes there. 

The analysis above and Figure~\ref{Radon_FB}, suggests the following improvement: we need to sample denser when $\beta$ is closer to $0$. In fact, one can set $p=R\sin\beta$ (the $R$ factor is not essential), as in \r{R1a} and sample uniformly in $p$. 
We will explore that route in a forthcoming paper. 

\subsection{(ii) Resolution limit given the sampling rate of $\mathcal{R}_\kappa f$}  
 Let $s_\alpha$, $s_\beta$ be the relative sampling rates in $\alpha$ and $\beta$, respectively. The Nyquist limit for $(\hat\alpha, \hat\beta)$ is given by $|\hat\alpha| <\pi/s_\alpha$, $|\hat\beta| <\pi/s_\beta$. By \r{R4}, this is equivalent to 
\be{RP3}
|x\cdot\xi^\perp|<\pi/s_\alpha, \quad \left| \pm \sqrt{ R^2|\xi|^2-(x\cdot\xi)^2 }+ x\cdot \xi^\perp\right| <\pi/s_\beta. 
\ee
Let $\theta$ be the angle which $\xi$ makes with $x$ when $x\not=0$, more precisely, $\theta$ is such that $x\cdot\xi=|x|\cos\theta$, $ x\cdot\xi^\perp= |x|\sin \theta$. Then   
\[
 |x||\xi| |\sin \theta| <\pi/s_\alpha, \quad |\xi| \left|\pm \sqrt{R^2-|x|^2\cos^2\theta}+ |x| \sin\theta\right|<\pi/s_\beta. 
\]
We plot the regions determined by the inequalities above with $s_\alpha=2s_\beta$, see \r{s_ab} to get the resolution diagram plotted on Figure~\ref{res_diagram}, where $R=1$ . 
\begin{figure}[h!] 
  \centering
    \includegraphics[trim = 0 120 0 120, clip ,height=.16\textheight, page=6]{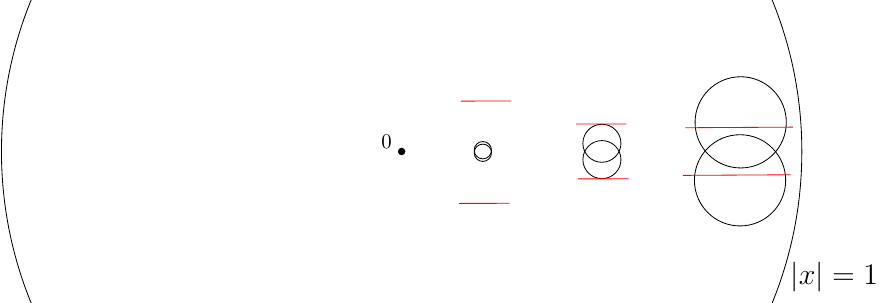}
\caption{\small  The resolution diagram of $\mathcal{R}_\kappa$ in fan-beam coordinates in the unit ball. For each $x$, the circles around it represent the frequency limit  imposed by $s_\beta$ as a function of the direction. A small radius means small frequency and therefore a smaller  resolution. The horizontal lines mark the resolution limit imposed by $s_\alpha$. 
The diagram is rotationally symmetric.}
\label{res_diagram}
\end{figure}
The horizontal lines represent the resolution limit imposed by $s_\alpha$. It is greatest near the origin and decreases (in vertical direction) away from the center. As it can be seen from \r{RP3}, the second inequality in \r{RP3} (satisfied for both signs) implies the first one; so that actual resolution is controlled by the double circles there except at $|x|=1$, where the lines are tangent to the lens shaped region.  
Next, the symmetry relation \r{symm2} has an interesting implication. Let $(\alpha_\pm, \beta_\pm, \hat\alpha_\pm, \hat\beta_\pm)$ be the image of $(x,\xi)$ under $C_\pm$, related by \r{symm2}, see Figure~\ref{fig_R4}.  Then the resolution limit on $f$ at $x$ in various directions posed by the sampling rate $s_\beta$ near $(\alpha_+, \beta_+)$ \textit{and}  near $(\alpha_-, \beta_-)$ are   given by \r{RP3} with both choices of the signs $\pm$. On Figure~\ref{res_diagram}, they are represented by the intersection of the two disks at each $x$. 
We can see that near the origin, it is quite small and close to isotropic. Near $|x|=1$, the resolution increases and it is better for $\xi$ close to radial (for example, for circular lines). On the other hand, since $C$ is $1$-to-$2$, we need only one of  $(\alpha_\pm, \beta_\pm, \hat\alpha_\pm, \hat\beta_\pm)$ to recover $(x,\xi)$. Therefore, the data $\mathcal{R}_\chi f$ actually contains stable information about recovery the singularities of $f$ in the union of those disks, instead of its intersection, if we can use that information.  It follows from \r{R4} that the better resolution is coming from that of the two lines through $x$ normal to $\xi$ with a source $R\omega(\alpha)$ which is closer to $x$. 
In the example in Figure~\ref{fig_R4}, for example, if the sampling rate of $\mathcal{R}_\kappa f$ is not sufficient to sample the right-hand pattern, we can just cut it off smoothly and use the other one only. 

Another approach is to note that
the shape the double triangles in Figure~\ref{fig_f_gaussian} allow for undersampling up to half of the rate, and when  there is aliasing (overlapped shifted triangles), it affects both images of every $(x,\xi)$ equally. To benefit from this however, 
  instead of using a sinc type of interpolation, we need to use $\chi$  in \r{2.3} with $\hat\chi$ supported in the double triangle in Figure~\ref{fig_f_gaussian}. Even better, we can sample on a parallelogram type of lattice as in \r{R4.2}.  Unlike \cite{Natterer-book, Natterer-sampling1993} we could have a non-uniform sampling set as in Section~\ref{sec_NU1} by dividing $[0,2\pi]\times [-\pi/2,\pi/2]$ into horizontal strips and using  parallelogram-like lattices in each one of varying densities using the fact that the wave front set size decreases when approaching $\beta=\pm\pi/2$. 
  
In Figure~\ref{fig_f_two_coherent} below, we present numerical evidence of this analysis. The phantom consists of two coherent states; each one a parallel transport of the other. Their wave front sets are localized in the $x$ and the $\xi$ variables. For the state on the left, we have $x$ almost parallel to $\xi$ on the wave front set, while for the state on the top, $x$ is almost perpendicular to $\xi$. As a result, the singularities of the first state are mapped to the lower frequency ones on the plot of $\mathcal{R}f$ closer to the corners. The state on the top creates the higher frequency oscillations of $\mathcal{R} f$ along the equatorial line of the plot of $\mathcal{R} f$. The Fourier transform on the right in Figure~\ref{fig_f_two_coherent} confirms that --- the four streaks closer to the borders correspond to the top phantom. Note that the horizontal axis in the Fourier transform plot is stretched twice compared to Figure~\ref{fig_f_gaussian} because the sampling requirement requires the same number of points on each axis, and then the discrete Fourier transform maps a square to a square. 

\begin{figure}[h!] 
  \centering\hfill 
\begin{subfigure}[t]{.2 \linewidth}
\centering\hfill 
\includegraphics[height=.15\textheight]{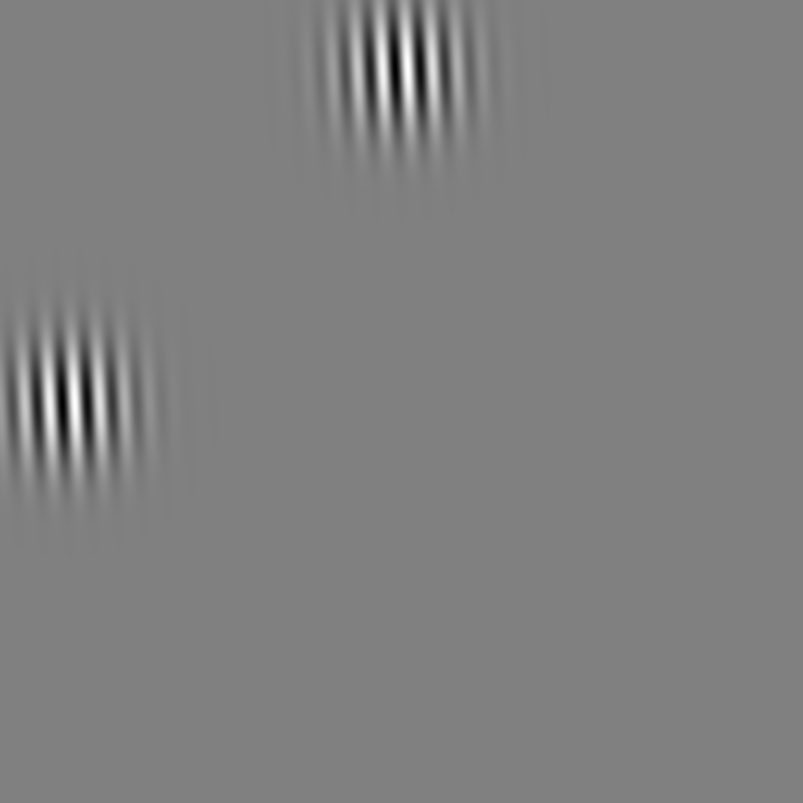}
\caption{$f$ on $[-1,1]^2$}\label{fig55a}
\end{subfigure}\hfill
\begin{subfigure}[t]{.2 \linewidth}
\centering
\includegraphics[height=.15\textheight]{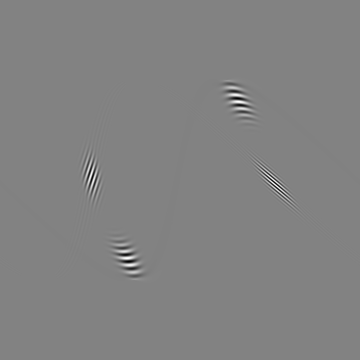}
\caption{$\mathcal{R}f$}\label{fig55b}
\end{subfigure}\hfill
\begin{subfigure}[t]{.2 \linewidth}
\centering
\includegraphics[height=.15\textheight]{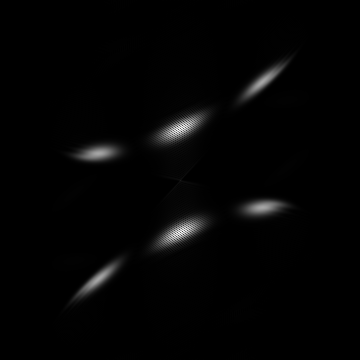}
\caption{$|\mathcal{FR}f|$}\label{fig55c}
\end{subfigure}\hfill\, 
\caption{\small (a): $f$ having $\WFH(f)$ at two  points.  (b): $\mathcal{R}f$ in fan-beam coordinates. (c): $|\mathcal{F}_h \mathcal{R}f|$. The two  patterns on the central horizontal line correspond to the phantom on the top. They are close to be critically sampled. The two patterns on the diagonal correspond to the phantom on the right.}
\label{fig_f_two_coherent}
\end{figure}

\subsection{(iii) Aliasing artifacts} 
If $\mathcal{R}_\kappa f$ is undersampled in either variable, we would get aliasing artifacts as h-FIOs related to shifts of $\hat\alpha$ and $\hat\beta$, see \r{sh},  \r{A9} Section~\ref{sec_alias_FIO}. By \r{R4cc} this would create shifts in the $x$ variable along $\xi^\perp$ and a possible change of the magnitude of $\xi$ but not its direction. We observed similar effects in the parallel parameterization case. In  Figure~\ref{fig_R_FB_AA}b one can see that the aliasing artifacts are extended outside the location of the ``doughnuts'' there.

\subsection{(iv) Averaged measurements} As in section~\ref{sec_R_AI}, assume we measure $Q_h\mathcal{R}_\kappa f$ with $Q_h$ an \HPDO\ of order $(0,0)$. If $q_0$ is the principal symbol of $Q_h$, then a backprojection reconstructs $P_hf$ with $P_h$ having principal symbol as in \r{RAI1}. In particular, if $q_0=\psi(a|\hat\alpha|^2+b|\hat\beta|^2)$, then 
\be{RFB1}
\begin{split}
q_0 &= \frac12\psi \Big(a |x\cdot \xi^\perp|^2 + b \big|x\cdot \xi^\perp+ \sqrt{R^2 |\xi|^2-(x\cdot \xi)^2}\big|^2\Big)\\
 &\quad {}+ \frac12\psi \Big(a |x\cdot \xi^\perp|^2 + b \big|x\cdot \xi^\perp -  \sqrt{R^2 |\xi|^2-(x\cdot \xi)^2}\big|^2\Big).
\end{split}
\ee
This formula reveals something interesting, similar to the observations above: the loss of resolution coming from each term is different. For each $(x,\xi)$, the reconstructed $f$ is $q_0(x,hD)f$ plus a lower order term, which is a sum of two with different (and direction dependent) losses of resolution. Let us say that $\psi$ is radial and decreasing as $r=|x|$ increases.  If $x\cdot\xi^\perp>0$, then the first term attenuates at that frequency more than the second one, and vice versa. Therefore, the reconstruction with full data in specific regions and directions would have less resolution that one with partial data. This is also illustrated in  Figure~\ref{res_diagram}: the intersection of the circles there reflects the resolution limit if we use full data and the union --- the resolution limit with partial data chosen to maximize the resolution. To take advantage of that, we would need to take a \HPDO\ $Q_h$, not just a convolution. 

In Figure~\ref{fig_R_FB_average} we show an example. In (b), we show the reconstructed $f$ with $\mathcal{R}f$ averaged in $\alpha$. 
\begin{figure}[h!]{} 
\hfill
\begin{subfigure}[t]{.2 \linewidth}
\centering
\includegraphics[height=.15\textheight]{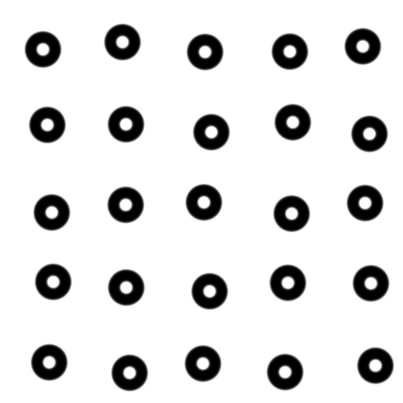}
\caption{$f$ on $[-1,1]^2$}\label{fig6a}
\end{subfigure}\hfill
\begin{subfigure}[t]{.2 \linewidth}
\centering
\includegraphics[height=.15\textheight]{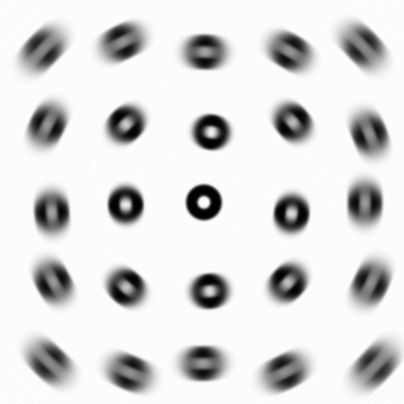}
\caption{$f_\text{rec}$ with $\mathcal{R}f$ averaged in $\alpha$}\label{fig6b}
\end{subfigure}\hfill
\begin{subfigure}[t]{.2 \linewidth}
\centering
\includegraphics[height=.15\textheight]{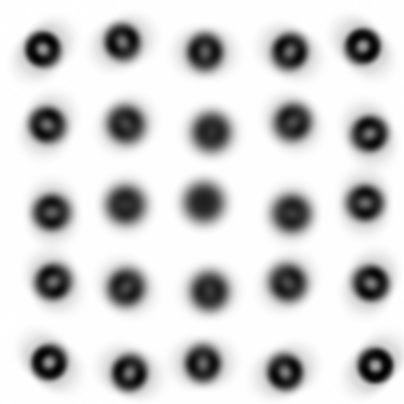}
\caption{$f_\text{rec}$ with $\mathcal{R}f$ averaged in $\beta$}\label{fig6c}
\end{subfigure}
\hfill 
\begin{subfigure}[t]{.2 \linewidth}
\centering
\includegraphics[height=.15\textheight]{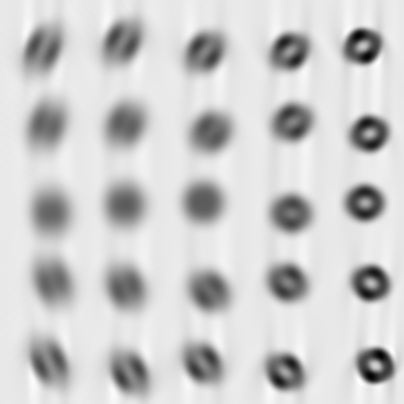}
\caption{$f_\text{rec}$ with $\mathcal{R}f$ averaged in $\beta$ and half-data}\label{fig6d}
\end{subfigure}\hfill\, \ \  
\caption{$f$ and a reconstructed $f_\text{rec}$ on $[-1,1]^2$ with data on the circumscribed circle and  data averaged in the $\alpha$ and in the $\beta$ variable. In (d), we use $\mathcal{R}f$ with sources on the right-hand part of the circle. 
}  \label{fig_R_FB_average}
\end{figure}
This corresponds to $b=0$ in \r{RFB1}. The worst resolution is where $|x\cdot\xi^\perp|$ is maximized, which happens when $|x|$ is maximized and $x\perp \xi$, like for radial lines close to the boundary. We have the best resolution when $|x\cdot\xi^\perp|$ is small, and if we want that for all directions; this happens near the origin but circular lines away from the origin are resolved well, too. Averaging in $\beta$ is represented by (c) and corresponds to $a=0$ in \r{RFB1}. As explained above, we get a superposition of two images and the understand the plot better, one should look first at (d), where a reconstruction with $\alpha$ restricted to $[-\pi/2,\pi/2]$ (the r.h.s.\ of the circumscribed circle) is shown. There, for ``doughnuts'' closer to the right-hand side, radial lines (where $\xi\perp  x$) are resolved better than circular ones, which corresponds to the union of the disks in Figure~\ref{res_diagram}. On the left (far from the sources), it is the opposite: radial lines are very blurred, while circular ones are better resolved. This corresponds to the intersection of the disks in Figure~\ref{res_diagram} which predicts better resolution for $\xi\parallel x$. Then in (c), we have a superposition of two such images which have a combined resolution in which radial and circular blur are mixed: there is still better resolution of radial lines (but the effect is subtle in this example) and a larger radius blur in circular directions. The effect is stronger near the corners as compared to ``doughnuts'' near the edges but in the center of each side because the former are closer to the circumscribed circle. 
\begin{figure}[h!]{} 
\hfill
\begin{subfigure}[t]{.2 \linewidth}
\centering
\includegraphics[height=.18\textheight]{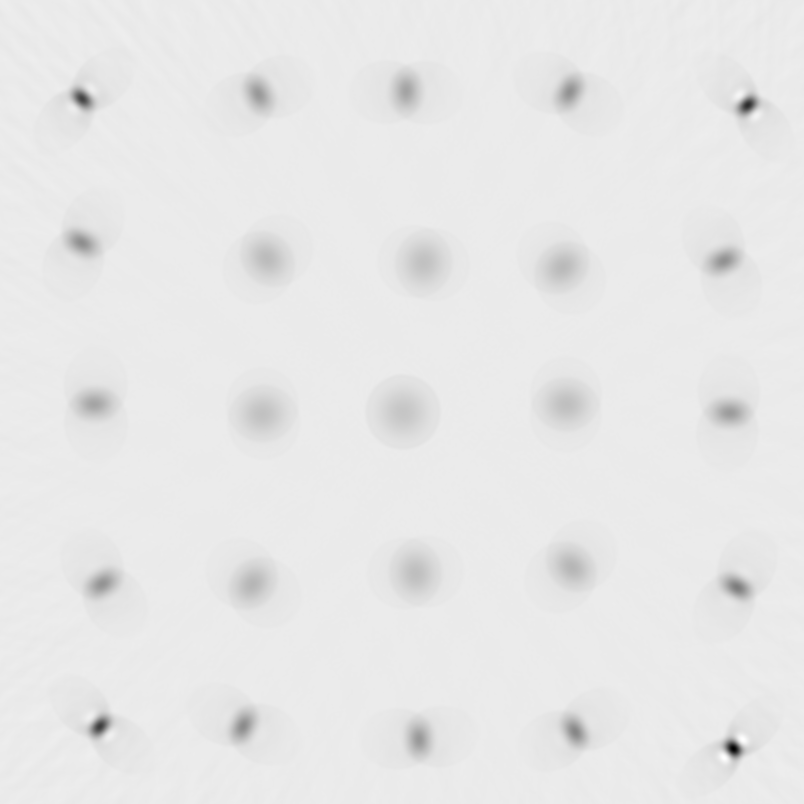}
\caption{$f$ reconstructed on $[-1,1]^2$}\label{fig7a}
\end{subfigure}\hfill
\begin{subfigure}[t]{.25 \linewidth}
\centering
\includegraphics[height=.18\textheight]{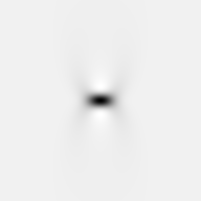}
\caption{the predicted blur kernel at $R(0.88,0)$}\label{fig7b}
\end{subfigure}\hfill\, \ \  
\caption{$f$ is a $5\times 5$ array of almost point-like Gaussians.  (a) the reconstructed $f$ on $[-1,1]^2$ with data on the circumscribed circle with $R=1.45$   averaged in  the $\beta$ variable.; (b) the predicted blur kernel at $R(0.88,0)\approx (1.28,0)$, enlarged. 
}  \label{fig_R_FB_average2}
\end{figure}

To illustrate this effect even better in Figure~\ref{fig_R_FB_average2}, we take $f$ to be a slightly randomized  $5\times 5$ array of very well concentrated Gaussians and apply a Gaussian blur to $\mathcal{R}f$ in $\beta$. The reconstruction is shown in Figure~\ref{fig_R_FB_average2}a. 
Then we compute numerically $\mathcal{F}^{-1}q_0$  see \r{RFB1} with $a=0$, which represents the  convolution kernel  of the reconstructed image at $x=R(0.88,0)$,  treating $x$ as a constant. We plot it  (enlarged) in (b). 
This is what the theory predicts to be the reconstructed image of a delta placed at that $x$. We can see a strong horizontal (i.e., radial) blur plus a fainter vertical (circular) one, spread over a larger area,  with a negative sign.
 In this  grayscale, black corresponds to the maximum and white corresponds to the minimum.  
In 
(a), 
one can see (smaller) similar images in the four corners, which are close to the circumscribed circle. Their orientations are along the radial lines, of course. As $x$ moves closer to the center, the kernel looks more circularly symmetric and gets larger, which can be seen from Figure~\ref{fig_R_FB_average2}a and also from \r{RFB1}. At the origin, it is Gaussian as \r{RFB1} predicts.

\textbf{Anti-aliasing.} In Figure~\ref{fig_R_FB_AA}, we present an example of $f$ undersampled in the $\beta$ variable and them blurred fist (in the same variable) and still undersampled at the same rate.

\begin{figure}[h!]{} 
\hfill
\begin{subfigure}[t]{.2 \linewidth}
\centering
\includegraphics[height=.15\textheight]{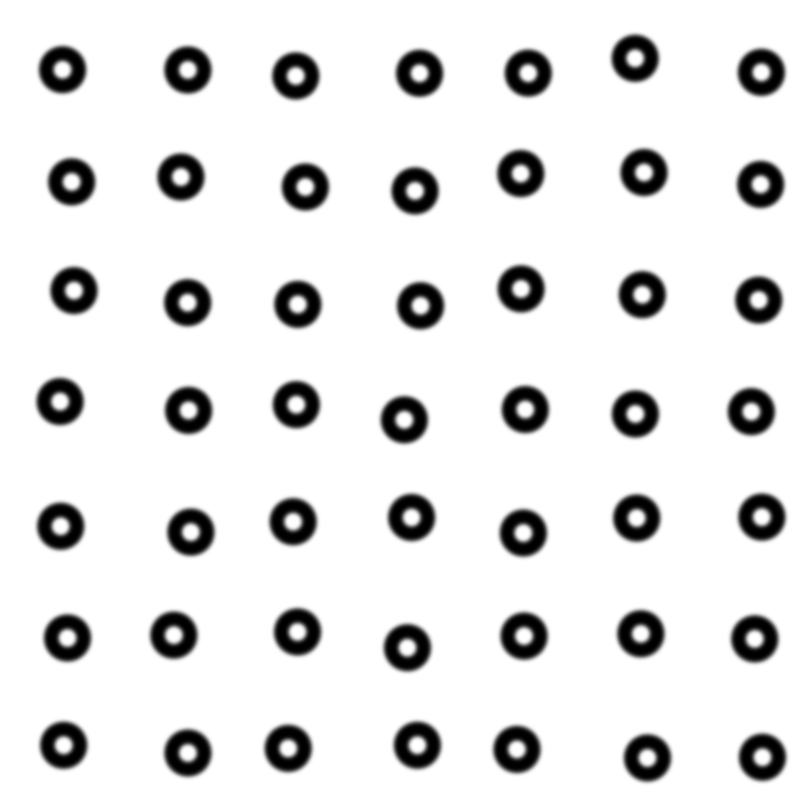}
\caption{original $f$  on $[-1,1]^2$}\label{fig8a}
\end{subfigure}\hfill
\begin{subfigure}[t]{.2 \linewidth}
\centering
\includegraphics[height=.15\textheight]{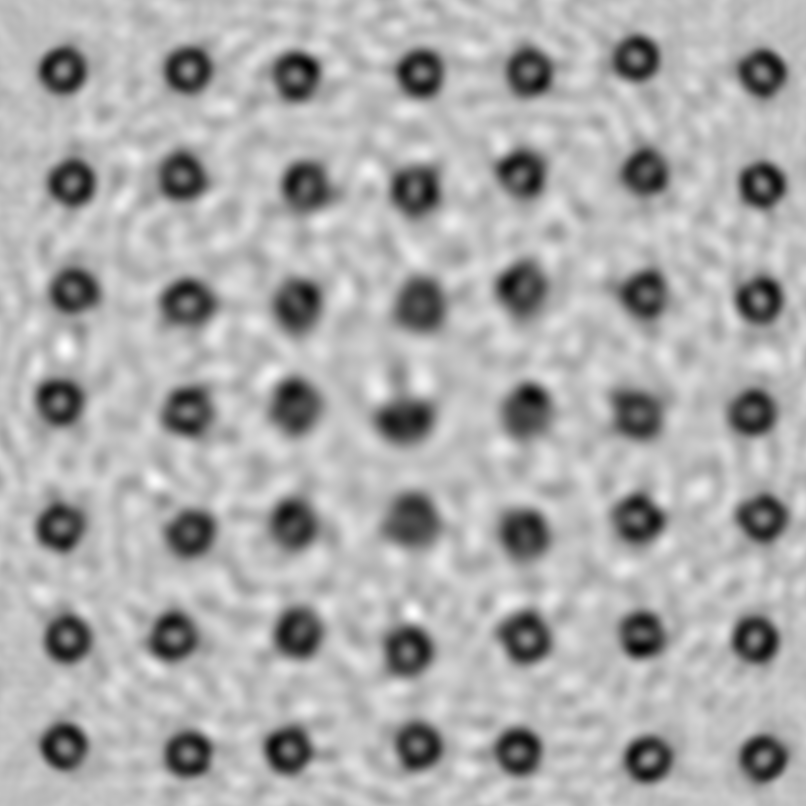}
\caption{a reconstruction $\mathcal{R}f$ undersampled in $\beta$}\label{fig8b}
\end{subfigure}\hfill 
\begin{subfigure}[t]{.2 \linewidth}
\centering
\includegraphics[height=.15\textheight]{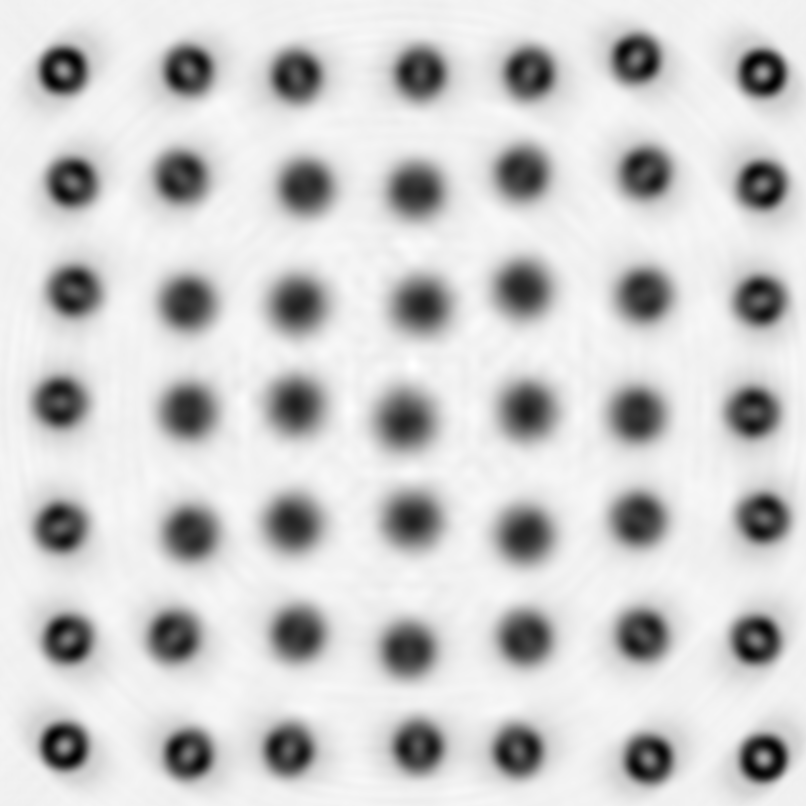}
\caption{an anti-aliased reconstruction}\label{fig8c}
\end{subfigure}\hfill\, \ 
\caption{$f$ is a $7\times 7$ array of ``doughnuts''.  (b) The reconstructed $f$  with data undersampled in  the $\beta$ variable, with $90$ angles:  a step  $s_\beta=2^o$; (b) $f$ is blurred first and then sampled as in (b). 
}  \label{fig_R_FB_AA}
\end{figure}

We see that the aliasing artifacts are mostly suppressed but some resolution is lost. 

\section{Thermo and Photo-Acoustic Tomography} 
Let $\Omega$ be a smooth bounded domain in $\R^n$. 
Let $g_0$ be a  Riemannian metric in $\bar\Omega$,   and let  $c>0$ be smooth.  Assume that $c=1$ and $g_0$ is Euclidean on $\bo$ (not an essential assumption). 
Fix $T>0$. Let $u$ solve the problem
\begin{equation}   \label{TAT1}
\left\{
\begin{array}{rcll}
(\partial_t^2 -c^2\Delta_{g_0} )u &=&0 &  \mbox{in $(0,T)\times \R^n$},\\
u|_{t=0} &=& f,\\ \quad \partial_t u|_{t=0}& =&0.
\end{array}
\right.               
\end{equation} 
Here, $\partial_\nu = \nu^j \partial_{x^j}$, where $\nu$ is the unit 
outer normal vector field on $\bo$.  The function  $f$ is the source which we eventually want to recover. The Neumann boundary conditions correspond to a ``hard reflecting'' boundary $\bo$. In applications, $g_0$ is Euclidean   but the speed $c$ is variable. The analysis applies to more general second order symmetric operator involving a magnetic field and an electric one, as in \cite{SU-thermo}. The metric determining the geometry is $g:=c^{-2}g_0$. We assume that $\bo$ is convex. 

Let $\Gamma\subset\bo$ be a relatively open subset of $\bo$, where the measurements are made. The observation operator is then modeled by
\be{I1}
\Lambda f = u|_{[0,T]\times\Gamma}.
\ee
The inverse problem is to find $f$ given $\Lambda f$.

The natural space for $f$ is the Dirichlet space $H_{D}(\Omega)$ defined as the completion of $C_0^\infty(\Omega)$ under the Dirichlet norm
\be{2.0H}
\|f\|_{H_{D}}^2= \int_\Omega |\nabla u|_g^2 \,\d\Vol.
\ee

The model above assumes acoustic waves propagating freely through $\bo$ where we make measurements. This means that the detectors have to be really small so that we can ignore their size. A different model studied in the literature is to assume that the waves are reflected from the boundary and measured there. To be specific, we may assume zero Neumann conditions on $\bo$ and then $\Lambda f$ would be the Dirichlet data but other combinations are possible. Then we solve first 
\begin{equation}   \label{TAT2}
\left\{
\begin{array}{rcll}
(\partial_t^2 -c^2\Delta_{g_0})u &=&0 &  \mbox{in $(0,T)\times \Omega$},\\
\partial_\nu u|_{(0,T)\times\bo}&=&0,\\
u|_{t=0} &=& f,\\ \quad \partial_t u|_{t=0}& =&0
\end{array}
\right.               
\end{equation} 
and define $\Lambda f$ as in \r{I1} again but this time $u$ is different. 

As shown in \cite{SU-thermo}  
in the first case \r{TAT1}, 
$\Lambda$, restricted to $f$ supported (strictly) in $\Omega$, is an elliptic  FIO of order zero with a canonical relation $C=C_-\cup C_+$, where
\be{TAT-CR}
C_\pm :(x,\xi) \longmapsto\Bigg(\underbrace{\pm s_\pm(x,\xi/|\xi|_g)}_t, \underbrace{\gamma_{x,\xi}( s_+(x,\xi))}_y,  \underbrace{\mp|\xi|_g}_{\hat t=\tau },\underbrace{ \dot\gamma'_{x,\xi}(s_\pm(x,\xi))}_{\hat y= \eta}\Bigg),
 \ee
with $s_\pm(x,\xi)$ being the exit time of the   geodesic starting from $x$ in the direction $\pm g^{-1}\xi$ (this is $\xi$ identified as a vector by the metric $g$) until it reaches $\bo$. We assume that $c^{-2}g_0$ is non-trapping; them those exit times are finite and positively homogeneous in $\xi$ of degree $-1$. Also, $\dot\gamma'$ stands for the orthogonal (in the metric) projection of $\dot\gamma$ to $T\bo$. Clearly, the frequency range of $C$ is  the space-like cone $|\eta|< |\tau|$. The norm $|\xi|_g$ is the norm of $\eta$ as a covector in the metric $g$, and similarly, $|\eta|$ is in the metric on $\bo$ induced by the Euclidean one on $\R^n$. We would have equality if $\xi$ is tangent to $\bo$ but this cannot happen since $\supp f\subset\Omega$. 
 
If we use \r{TAT2} as a model instead (allowing for reflections) it was shown in 
\cite{St-Y-AveragedTR} that the first singularities give rise to an FIO $\Lambda$ with the same canonical relation, which is actually $2\Lambda$ modulo a lower order operator. After each reflection, we get an FIO with a canonical relation of the same type but reflected from the boundary. The sampling requirements are the same, and we will skip the details.

Assume now that $\Sigma_h(f)\subset \{  |\xi|\le B\}$.  We have  $|\xi|_g^2= c^2g_0^{ij}\xi_i\xi_j$. Let $M^2$ be the sharp lower bound of the metric form $c^{-2}g_0$ on the unit sphere over all $x$. Then $1/M$ is the sharp upper bound on $c^2g^{-1}$ and $|\xi|_g\le B/M$ which is sharp. Then  
\[
\Sigma_h(\Lambda f)  \le \{ (\tau,\eta)\in \R\times T^*\bo; \; |\eta|<|\tau|\le B/M\}
\]
\begin{wrapfigure}[11]{r} {0.46\textwidth}
\begin{center}
	\includegraphics[page=2 ]{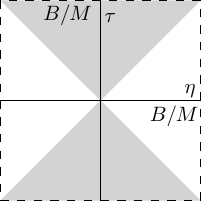}
\end{center}
\caption{\small The frequency set of $\Lambda f$.}
\label{fig_TAT0}
\end{wrapfigure}
and the r.h.s.\ is actually the range of $\Sigma_h (\Lambda f)$ for all $f$ as above, see Figure~\ref{fig_TAT0}. If we sample 
 on a grid on $[0,T]\times\bo$, with the second variable in a fixed coordinate chart, we need to choose steps $\Delta t<\pi M h/B$ and $|\Delta y^j|<\pi Mh/B$, where the latter norm is in the induced metric. Since $\Delta y^j$ is constant (the superscript $j$ refers to the $j$-th coordinate), for the Euclidean length $\Delta y^j$ we must have $\Delta y^j<\pi M'Mh/B$, where $(M')^2$ is the sharp upper bound on the induced metric on the Euclidean sphere in that chart. In our numerical example below, the boundary is piecewise flat parameterized in an Euclidean way; then $M'=1$ away from the corners. 
 
Set $c_\text{max}=\max c$. If $g$ is Euclidean,  $M=1/c_\text{max}$. The metric on $\R\times\bo$ is $\d t^2 +g'$, where $g'$ is the Euclidean metric restricted to $T\bo$. The sampling requirements in any local coordinates on the boundary depend in those coordinates as explained above, with $M'=1$.  Therefore, the sampling rate in the $(t,y)$ coordinates should be smaller than $\pi h/B c_\textrm{max}$. 

It is interesting that the sampling requirements do not depend on existence of conjugate points or not and are unaffected by possible presence of caustics. In fact, we can have caustics even if the geometry is Euclidean but we start from a concave wave front. 
\begin{figure}[h!] 
  \centering
	\includegraphics[trim = 0 0 0 0 , clip,height=.12\textheight]{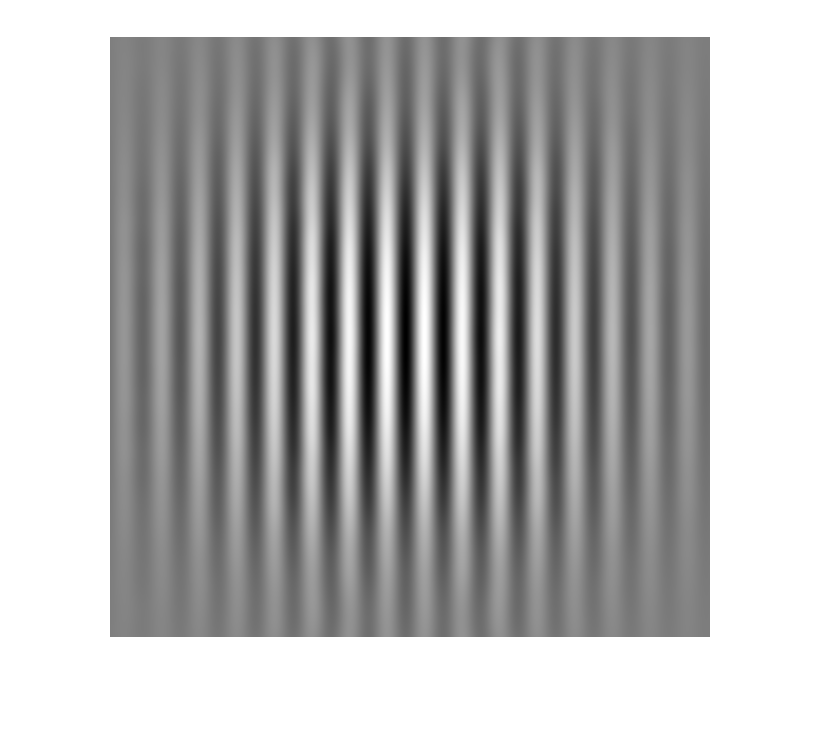}\ \ \ \ \ 
   \includegraphics[trim = 0 0 0 0 , clip, height=.12\textheight]{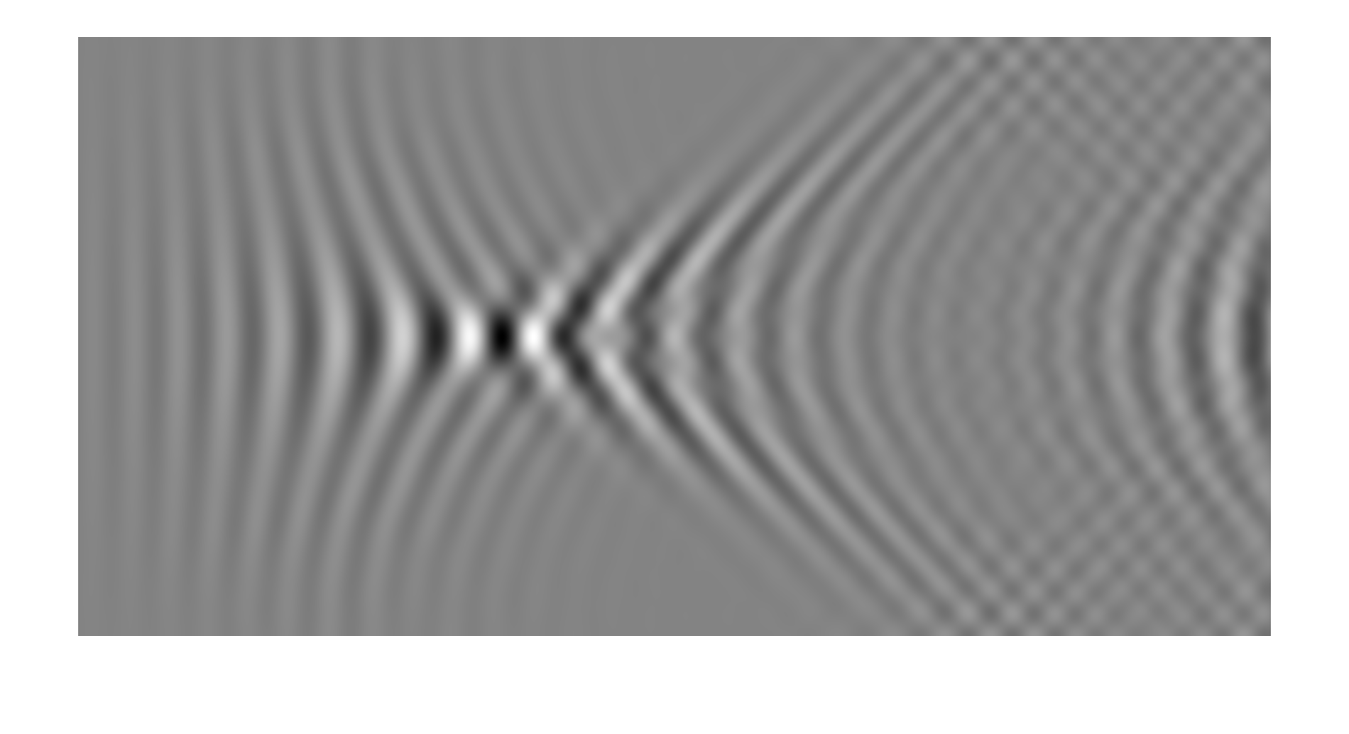}\ \ \ \ 
   \includegraphics[trim = 0 3.2 0 0 , clip, height=.129\textheight]{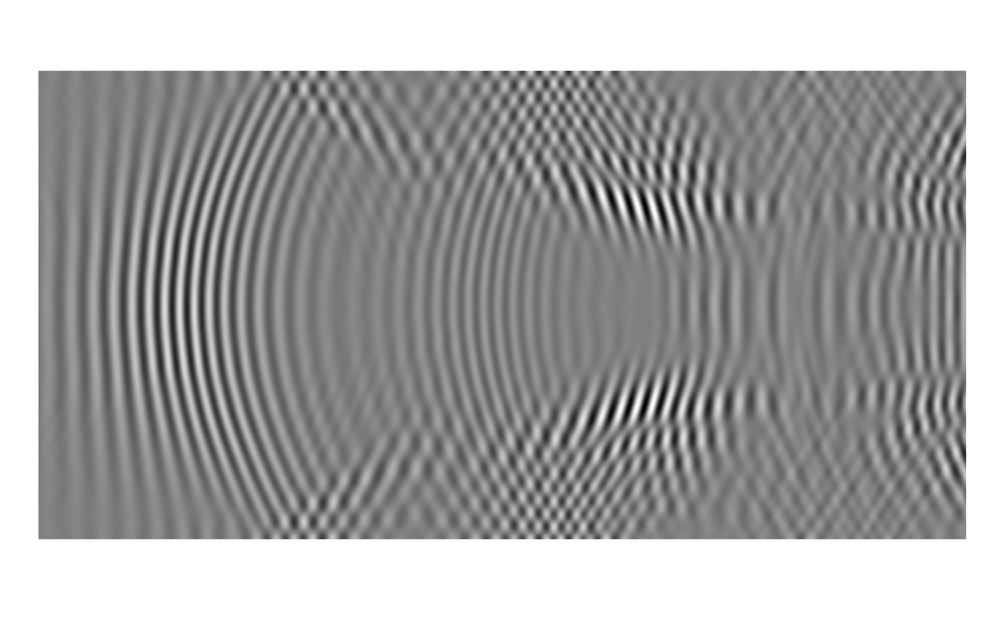}
\caption{\small Left: $f$ having $\WFH(f)$ along horizontal directions, on the $[-1,1]^2$ square.  Center: $\Lambda f$ on the right hand side cross time for $0\le t\le 4$ with a speed with  a slow region in the center. Despite the presence of caustics,  $\Lambda f$ does not contain higher frequencies there. Right: $\Lambda f$ with a speed having a fast region in the center.}
\label{TAT_pic1}
\end{figure}
In Figure~\ref{TAT_pic1} on the left, we plot  $f= e^{-2|x|^2}\sin((x-0.3)/0.02)$ in the square $[-1,1]^2$ computed with a high enough resolution. On the right, we plot $\Lambda f$ for the second model \r{TAT2} on the right hand side of the square cross the time interval $[0,4]$. The speed is $c=1-0.5\exp(-2|x|^2)$ having a slow region in the center and range $0.5\le c\le 1-e^{-2}/2\approx 0.93$. 
Then $M\sim 1$, and as noticed above, $M'=1$. The sampling requirements of $\Lambda f$ on $[0,4]\times\bo$ are therefore  the same as those of $f$ on $[-1,1]^2$.  
 Figure~\ref{TAT_pic1} demonstrates that fact by showing that the highest frequencies of $\Lambda f$ in the center are approximately the same as the highest ones on the left. Naturally, they occur where the rays hit $\bo$ at the largest  angle with the normal which is represented by the slanted curves on the plot of $\Lambda f$.  Next, despite of presence of caustics a bit left of the center of the $\Lambda f$ plot,  the oscillations are not of higher frequencies than elsewhere else. On the right, we plot $\Lambda f$  when $c=1+0.5\exp(-2|x|^2)$, i.e., there is a fast region in the middle. The speed range is approximately $[1.07,1.5]$. There are higher frequencies than in the previous case and higher than in $f$. The sampling requirements are higher. 

A more thorough analysis of this case in the context of this paper will be presented elsewhere.

\end{document}